\documentclass[oneside,11pt,reqno]{amsart}

\usepackage[margin=1.1in]{geometry}

\usepackage{amssymb}
\usepackage{braket}
\usepackage{mathrsfs}
\usepackage{ifthen}
\usepackage{here}
\usepackage{todonotes}
\usepackage{tikz}
\usetikzlibrary{patterns}
\usepackage{comment} 
\usepackage{mleftright}
\usepackage{cleveref}
 \usepackage[all]{xy}
 \usepackage{amscd}
 \usepackage{amsrefs}
 \usepackage{color}
 \usepackage{enumitem}
\newlist{steps}{enumerate}{1}
\setlist[steps, 1]{label = Step \arabic*:}
\usepackage[abs]{overpic}
\usepackage[latin1]{inputenc}
\usepackage{tikz-cd}

\makeatletter
\DeclareRobustCommand\widecheck[1]{{\mathpalette\@widecheck{#1}}}
\def\@widecheck#1#2{%
   \setbox\z@\hbox{\m@th$#1#2$}%
   \setbox\tw@\hbox{\m@th$#1%
      {%
         \vrule\@width\z@\@height\ht\z@
         \vrule\@height\z@\@width\wd\z@}$}%
   \dp\tw@-\ht\z@
   \@tempdima\ht\z@ \advance\@tempdima2\ht\tw@ \divide\@tempdima\thr@@
   \setbox\tw@\hbox{%
      \raise\@tempdima\hbox{\scalebox{1}[-1]{\lower\@tempdima\box\tw@}}}%
   {\ooalign{\box\tw@ \cr \box\z@}}}
\makeatother

\theoremstyle{plain}
\newtheorem{thm}{Theorem}[section]
\crefname{thm}{Theorem}{Theorems}
\Crefname{thm}{Theorem}{Theorems}
\newtheorem{prop}[thm]{Proposition}
\crefname{prop}{Proposition}{Propositions}
\Crefname{prop}{Proposition}{Propositions}
\newtheorem{lem}[thm]{Lemma}
\crefname{lem}{Lemma}{Lemmas}
\Crefname{lem}{Lemma}{Lemmas}
\newtheorem{cor}[thm]{Corollary}
\crefname{cor}{Corollary}{Corollaries}
\Crefname{cor}{Corollary}{Corollaries}

\crefname{claim}{Claim}{Claims}
\Crefname{claim}{Claim}{Claims}

\crefname{property}{Property}{Properties}
\Crefname{property}{Property}{Properties}

\crefname{problem}{Problem}{Problems}
\Crefname{problem}{Problem}{Problems}

\crefname{conjecture}{Conjecture}{Conjecture}
\Crefname{conjecture}{Conjecture}{Conjecture}

\theoremstyle{definition}
\newtheorem{defn}[thm]{Definition}
\crefname{defn}{Definition}{Definitions}
\Crefname{defn}{Definition}{Definitions}

\crefname{notation}{Notation}{Notations}
\Crefname{notation}{Notation}{Notations}

\crefname{convention}{Convention}{Conventions}
\Crefname{convention}{Convention}{Conventions}

\crefname{cond}{Condition}{Conditions}
\Crefname{cond}{Condition}{Conditions}

\crefname{assum}{Assumption}{Assumptions}
\Crefname{assum}{Assumption}{Assumptions}

\crefname{conj}{Conjecture}{Conjectures}
\Crefname{conj}{Conjecture}{Conjectures}

\crefname{claim1}{Claim}{Claims}
\Crefname{claim1}{Claim}{Claims}
\Crefname{ques}{Question}{Question}
\newtheorem{ques}[thm]{Question}
\crefname{que}{Question}{Question}
\Crefname{que}{Question}{Question}

\theoremstyle{remark}
\newtheorem{rem}[thm]{Remark}
\crefname{rem}{Remark}{Remarks}
\Crefname{rem}{Remark}{Remarks}

\crefname{ex}{Example}{Examples}
\Crefname{ex}{Example}{Examples}

\crefname{section}{Section}{Sections}
\Crefname{section}{Section}{Sections}
\crefname{subsection}{Subsection}{Subsections}
\Crefname{subsection}{Subsection}{Subsections}
\crefname{figure}{Figure}{Figures}
\Crefname{figure}{Figure}{Figures}

\newtheorem*{acknowledgement}{Acknowledgement}

\newcommand{\Z}{\mathbb{Z}}

\newcommand{\CP}{\mathbb{CP}}

\newcommand{\fraks}{\mathfrak{s}}
\newcommand{\frakt}{\mathfrak{t}}

\newcommand{\Diff}{\mathrm{Diff}}
\newcommand{\Homeo}{\mathrm{Homeo}}
\newcommand{\Aut}{\mathrm{Aut}}

\newcommand{\Map}{\mathrm{Map}}
\newcommand{\diag}{\mathrm{diag}}
\newcommand{\Spin}{\mathrm{Spin}}

\newcommand{\del}{\partial}

\newcommand{\im}{\operatorname{Im}}

\newcommand{\id}{\mathrm{id}}

\newcommand{\C}{\mathbb{C}}
\newcommand{\RP}{\mathbb{RP}}

\newcommand{\R}{\mathbb R}

\def\det{\operatorname{det}}

\def\dim{\operatorname{dim}}

\def\id{\operatorname{id}}

\newcommand{\mbar}[1]{{\ooalign{\hfil#1\hfil\crcr\raise.167ex\hbox{--}}}}

\def\Aut{\operatorname{Aut}}
\def\Map{\operatorname{Map}}

\def\wt{\widetilde}

     \RequirePackage{rotating}                   
    \def\HMt{%
       \setbox0=\hbox{$\widehat{\mathit{HM}}$}
       \setbox1=\hbox{$\mathit{HM}$}
       \dimen0=1.1\ht0
       \advance\dimen0 by 1.17\ht1
       \smash{\mskip2mu\raise\dimen0\rlap{%
          \begin{turn}{180}
              {$\widehat{\phantom{\mathit{HM}}}$}
           \end{turn}} \mskip-2mu    
                \mathit{HM}
                    }{\vphantom{\widehat{\mathit{HM}}}}{}}
\newcommand{\blue}[1]{\textcolor{blue}{#1}}

\title{Exotic Dehn twists on 4-manifolds}

\author{Hokuto Konno, Abhishek Mallick and Masaki Taniguchi}

\begin{document}

\maketitle

\begin{abstract}
We initiate the study of exotic Dehn twists along 3-manifolds $\neq S^3$ inside $4$-manifolds,
which produces the first known examples of exotic diffeomorphisms of contractible 4-manifolds, more generally of definite 4-manifolds, and exotic diffeomorphisms of 4-manifolds with $\neq S^3$ boundary that survive after one stabilization.
We also construct the smallest closed 4-manifold known to support an exotic diffeomorphism.
These exotic diffeomorphisms are the Dehn twists along certain Seifert fibered 3-manifolds.
As a consequence, we get loops of diffeomorphisms of 3-manifolds that topologically extend to some 4-manifolds $X$ but not smoothly so, implying the non-surjectivity of $\pi_1(\mathrm{Diff}(X)) \to \pi_1(\mathrm{Homeo}(X))$. 
Our method uses $2$-parameter families Seiberg--Witten theory over $\mathbb{RP}^2$, while known methods to detect exotic diffeomorphisms used $1$-parameter families gauge-theoretic invariants.
Using a similar strategy, we construct a new kind of exotic diffeomorphisms of 4-manifolds, given as commutators of diffeomorphisms.
\noindent
\end{abstract}


\section{Introduction}

\subsection{Background and overview}\label{1_1}

Comparison between the diffeomorphism group $\Diff(X)$ and homeomorphisms group $\Homeo(X)$ of a 4-manifold $X$ has a long history, starting with Friedman--Morgan~\cite{FM88} and Donaldson~\cite{Do90} in the study of the Donaldson invariant. Recently it became one of the active topics in 4-manifold topology through systematic applications of families gauge theory.
An {\it exotic diffeomorphism} represents one of the most basic discrepancy between $\Diff(X)$ and $\Homeo(X)$, where
a diffeomorphism $f : X \to X$ is said to be {\it exotic} if $f$ is topologically isotopic to the identity but not smoothly.
This is equivalent to $f$ representing a non-trivial element of the kernel of the map $\pi_0(\Diff(X)) \to \pi_0(\Homeo(X))$ induced from the forgetful map $\Diff(X) \to \Homeo(X)$, where $\Diff(X)$ and $\Homeo(X)$ are equipped with the $C^\infty$- and $C^0$-topology respectively.
If $X$ has non-empty boundary, we may consider a relative version of this notion, which is the main object of interest in this paper.
Let $\Diff(X,\del)$ denote the group of diffeomorphisms that fix $\del X$ pointwise, and similarly define $\Homeo(X,\del)$.
We say that $f \in \Diff(X,\del)$ is exotic, or {\it relatively exotic} (to avoid ambiguity), if $f$ is topologically isotopic to the identity through $\Homeo(X,\del)$, but not smoothly through $\Diff(X,\del)$.

The first examples of exotic diffeomorphisms of 4-manifolds were given by Ruberman~\cite{Ru98}, and recently several authors gave other examples \cite{BK18,KM20,JL20,IKMT22} including the relative case.
Amongst these, Kronheimer--Mrowka \cite{KM20} and Lin~\cite{JL20} proved that Dehn twists along $S^3$ is exotic on $K3\# K3$ and $K3\#K3\#S^2\times S^2$ (and on analogous examples for a homotopy $K3$ in place of $K3$ as well). These were the only known examples of exotic Dehn twists in the literature.
All of these known examples of exotic diffeomorphisms in the literature were detected by 1-parameter family versions of Donaldson/Seiberg--Witten/Bauer--Furuta invariants, or relative versions of such.
So far, tools to compute families gauge-theoretic invariants are quite limited, making it difficult to detect an exotic diffeomorphism of 4-manifolds with simple topology such as a definite or contractible 4-manifold. Moreover, for the same reason, it has not been possible so far to restructure the previous arguments to produce exotic Dehn twists along 3-manifolds $\neq S^3$.

In this article, we open up the study of exotic Dehn twists along 3-manifolds $\neq S^3$ on various types of 4-manifolds. The problem of whether the Dehn twist along a Seifert fibered 3-manifold may be smoothly non-trivial was discussed informally by Kronheimer and Ruberman during a talk by Kronheimer \cite[53rd Minute]{Krontalk}, where the authors learned this problem. 
More generally, after establishing the triviality of Dehn twists on simply-connected $4$-manifolds in the topological category, Orson--Powell~\cite[Question~1.2]{OrsonPowell2022} asked when a Dehn twist along a 3-manifold $\neq S^3$ may be smoothly non-trivial. 

In this article, we give the first non-trivial answer to these questions.
In fact, we will show that Dehn twists along Seifert fibered 3-manifolds produce numerous interesting exotic diffeomorphisms, including the first examples of:

\begin{enumerate}
\item[(i)] exotic diffeomorphisms of definite 4-manifolds (\cref{main thm1: Dehn twist on definite}),
\item[(ii)] exotic diffeomorphisms of contractible 4-manifolds (\cref{contractible}), and
\item[(iii)] exotic diffeomorphisms of 4-manifolds with $\neq S^3$ boundary that survive after one stabilization. (\cref{thm intro: one stabilization,thm: intro stabilization in K3}).
\end{enumerate}

A Dehn twist along a Seifert fibered 3-manifold also yields the smallest example of a closed 4-manifold that supports exotic diffeomorphisms (\cref{thm: closed Dehn}).
Further, the results (i), (ii), (iii) also yield examples of loops of diffeomorphisms of 3-manifolds that topologically extend to some 4-manifolds $X$ but not smoothly so (\cref{thm: pi1 1,cor: pi1 non surj contractible,thm pi 1 Milnor necleus}).
This implies the non-surjectivity of the natural map $\pi_1(\mathrm{Diff}(X)) \to \pi_1(\mathrm{Homeo}(X))$ (\cref{thm: pi1 smooth vs. top 1,cor: pi1 non surj contractible,thm pi 1 Milnor necleus}).


To prove the above results,
we introduce a new technique to detect exotic diffeomorphisms:
we will use families Seiberg--Witten theory over $\RP^2$ (\cref{Scheme of the proof}), instead of computing 1-parameter families gauge-theoretic invariants.
By deploying similar strategies, we give further new examples of exotic diffeomorphisms (\cref{More exotic diffeomorphisms: exotic commutators}), which we call the \textit{exotic commutator}. We describe our results in detail below:

\subsection{Dehn twists}

To state our results, let us recall the definition of Dehn twists, following \cite{OrsonPowell2022}.
Let $Y$ be a 3-manifold, and $\phi : S^1 \to \Diff(Y)$ be a loop  based at the identity which is non-trivial in $\pi_1(\Diff(Y))$. 
On the cylinder $Y \times [0,1]$, we can define a diffeomorphism $t_Y : Y \times [0,1] \to Y \times [0,1]$ by $(y,s) \mapsto (\phi(s)(y),s)$ for $(y,s) \in Y \times [0,1]$, where we regard $S^1=[0,1]/0 \sim 1$.
Since $t_Y$ lies in  $\Diff(Y \times [0,1], \del)$, if there is an embedding of the cylinder $Y \times [0,1]$ into a 4-manifold $X$, we can implant $t_Y$ in $X$ to get a diffeomorphism of $X$.
This diffeomorphism is called the {\it Dehn twist on $X$ along $Y$}, which we denote by $t_X : X \to X$. While $t_X$ depends on $\phi$, we drop $\phi$ from our notation.

In particular, if such a $Y$ is a boundary of a $4$-manifold $X$, we have a Dehn twist supported in a collar neighborhood of the boundary of $X$, which lies in $\Diff(X,\del)$.
We call this case the {\it boundary Dehn twist}. 
In particular, since a Seifert fibered space $Y$ admits a standard circle action (see \cref{Preliminary: definition of Dehn twists and Milnor fibers}) inducing a based loop in $\Diff(Y)$, one may consider the Dehn twist along a Seifert fibered space. From now on, the Dehn twist along a Seifert manifold will refer to the Dehn twist with respect to the canonical circle action of the Seifert fibered spaces.

\subsection{Exotic Dehn twists on definite/contractible 4-manifolds}
\label{Intro: Exotic Dehn twists}

The following \lcnamecref{main thm1: Dehn twist on definite} provides the first example of
a {\it definite} 4-manifold (with or without boundary) that supports an exotic diffeomorphism.
Recall that Brieskorn spheres $\Sigma(2,3,6n+7)$ bound simply-connected, positive-definite 4-manifolds, as they are $(+1)$-surgeries of certain (twist) knots. 
We have the following:

\begin{thm}
\label{main thm1: Dehn twist on definite}
Let $n$ be any integer with $n \geq 0$  and let $X$ be any compact, oriented, positive-definite, simply-connected, smooth 4-manifold bounded by $\Sigma(2,3,6n+7)$.
Then the boundary Dehn twist $t_X : X \to X$ is relatively exotic, i.e. $t_X$ is topologically isotopic to the identity through $\Homeo(X, \partial)$, but not smoothly isotopic to the identity through $\Diff (X, \partial)$.
\end{thm}

The existence of a topological isotopy is due to Orson and Powell~\cite{OrsonPowell2022} (this remark applies to all statements in the topological category in this paper). 
Our contribution is the non-existence of a smooth isotopy.

It is worth noting that every Seifert homology sphere is known to bound a negative-definite 4-manifold on which the boundary Dehn twist is smoothly trivial rel boundary (\cref{rem: equiv plumbing}).
\cref{main thm1: Dehn twist on definite} shows that a positive-definite analog of this fact does not hold.

As a consequence, we can detect exotic diffeomorphisms of contractible 4-manifolds.
Recall that $\Sigma(2,3,13)$ and $\Sigma(2,3,25)$ bound contractible (in fact Mazur) smooth 4-manifolds, by \cite{AK79} for $\Sigma(2,3,13)$ and by \cite{Fickle} for $\Sigma(2,3,25)$, respectively.

\begin{cor}
\label{contractible}
Let $C$ be any contractible, compact, smooth 4-manifold bounded by either $\Sigma(2,3,13)$ or  $\Sigma(2,3,25)$.
Then the boundary Dehn twist $t_C : C \to C$ is relatively exotic.
\end{cor}

\cref{contractible} gives the first example of a {\it contractible} 4-manifold that supports an exotic diffeomorphism.
Before this, the smallest known 4-manifold $X$ supporting an exotic diffeomorphism had $b_2(X)=4$ \cite[Remark~5.5]{IKMT22}.
Furthermore, \cref{main thm1: Dehn twist on definite} shows the Dehn twists along a general 3-manifold may survive after stabilizations by $\CP^2$, unlike the Dehn twist along $S^3$ (\cref{rem: stabilization by CP2}). 

As another infinite series of examples of exotic Dehn twists, we will see:

\begin{thm}
\label{thm: More Dehn twists}
Let $n$ be any integer with $n \geq 0$, the Brieskorn sphere $\Sigma(2,3,6n+11)$ bounds a simply-connected, oriented, compact, smooth 4-manifold $X$ for which the boundary Dehn twist $t_X$ is relatively exotic.
\end{thm}

In \cref{subsection 279}, we will also discuss a method to produce further examples different from the above families, such as $\Sigma (2,7,9)$.
As mentioned in \cref{1_1}, the Theorems above also provide the first non-trivial answers to the question posed by Orson--Powell~\cite[Question~1.2]{OrsonPowell2022} (see also Kronherimer's talk \cite{Krontalk}).


Before moving on to the next topic, we record a question related to the 4-dimensional Smale conjecture, which we learned from Kronheimer's talk \cite{Krontalk} and which has motivated us to study exotic Dehn twists on contractible 4-manifolds.
If we take a contractible 4-manifold $C$ bounded by $\Sigma(2,3,13)$ or $\Sigma(2,3,25)$ to be a Mazur manifold following \cite{AK79,Fickle}, the double $D(C)$ of $C$ is diffeomorphic to $S^4$.

\begin{ques}[{{{\it cf.} Kronheimer \cite{Krontalk}}}]
\label{ques S4}
Is the Dehn twist on $S^4$ along $\Sigma(2,3,13)$ and/or $\Sigma(2,3,25)$ smoothly isotopic to the identity?    
\end{ques}

If one solves Question~\ref{ques S4} in the negative, it will give an alternative disproof of the 4-dimensional Smale conjecture,  different from Watanabe~\cite{Wa18}.
Note that, if the Dehn twist on $S^4=D(C)$ were not isotopic to the identity,
then the boundary Dehn twist on $C$ must not be relatively isotopic to the identity, which has been confirmed in {\cref{contractible}}.
Thus, {\cref{contractible}} may be regarded as partial evidence suggesting that this alternative way to disprove the Smale conjecture may work.

\begin{rem}
Readers may notice that all the examples of Dehn twists above are along Seifert fibered spaces. Indeed, as a consequence of work by Gabai \cite{gabai2001smale} hyperbolic $3$-manifolds do not admit a non-trivial loop of diffeomorphism to twist along. However, our obstruction to detect the exotic nature of Dehn twists, in principle, can be applied to twists along any integer homology spheres (whenever the twists exist). 
\end{rem}

\subsection{One stabilization is not enough for exotic diffeomorphism with $\del \neq S^3$}
\label{subsection intro One stabilization}

As is well known, most of the exotic phenomena for simply-connected 4-manifolds disappear after taking connected sums with finitely many copies of $S^2 \times S^2$.
Until recently, there was no known example of exotic phenomena for which one stabilization (i.e. connected sum of a single copy of $S^2 \times S^2$) is not enough to kill.
However, within the past few years, several authors have established theorems of the form `one stabilization is not enough' for exotic 4-manifolds, exotic embeddings, and exotic diffeomorphisms (including relative versions 
 of these questions)\cite{JL20,LM21,Kang2022,KMT22,guth2022exotic,hayden2023atomic,HaydenKangMukherjee23,KMT23first}.
For example, any relatively exotic diffeomorphism is smoothly isotopic to the identity relative to the boundary after sufficient stabilizations, as far as the boundary is (for example) a homology 3-sphere (\cite[Theorem 3.7]{Saeki06}, \cite[Theorem~B]{OrsonPowell2022}).
Thus it makes sense to ask:
\begin{ques}
Given a relatively exotic diffeomorphism $f : X \to X$ of a simply-connected compact 4-manifold $X$ with homology 3-sphere boundary, 
how many stabilizations are enough to make $f$ isotopic to the identity rel boundary?
For example, how many stabilizations are enough to make the boundary Dehn twist trivial in the relative mapping class group?
\end{ques}
Again, the authors were reminded of this stabilization problem on the boundary Dehn twist in Kronheimer's talk \cite{Krontalk}. This problem was discussed also by Krannich--Kupers in \cite{krannich2022torelli} later.

The first and only known example of an exotic diffeomorphism for which `one stabilization is not enough' is the Dehn twist along $S^3$ in the neck of $K3\#K3$, and more generally the Dehn twist along $S^3$ on a sum of two homotopy $K3$, due to Lin~\cite{JL20} by a computation of the families Bauer--Furuta invariant.
In particular, this implies that the boundary Dehn twist of the punctured $K3$ survives after one stabilization.
We will provide the first example of exotic diffeomorphism (in the form of exotic Dehn twists) of $4$-manifolds with $\partial \neq S^3$ which survive after one stabilization.

To state the result, let $M(p,q,r)$ denote the Milnor fiber, whose boundary is given by $\del M(p,q,r) = \Sigma(p,q,r)$. 
For a 4-manifold $X$ with boundary, let $X\# S^2 \times S^2$ denote the inner connected sum of $S^2 \times S^2$.
We will prove:

\begin{thm}
\label{thm intro: one stabilization}
Let $X$ be either the Milnor fiber $M(2,3,7)$ or $M(2,3,11)$.
Then the boundary Dehn twist $t_X : X \to X$ is relatively exotic.
Moreover, $t_X \# \id_{S^2 \times S^2} : X \# S^2 \times S^2 \to X \# S^2 \times S^2$ is still relatively exotic.
\end{thm}

The assertion of 
\cref{thm intro: one stabilization} for $M(2,3,7)$, whose intersection form is $-E_8\oplus 2\begin{pmatrix}
0 & 1\\
1 & 0
\end{pmatrix}$, also gives the smallest 4-manifold which has been shown to support exotic diffeomorphisms surviving after one stabilization.
Before this, the smallest known such 4-manifold was $\mathring{K3}$, the punctured $K3$, deduced from Lin's result \cite{JL20}.

In contrast to \cref{thm intro: one stabilization}, any Seifert homology sphere is known to bound a complex surface on which the boundary Dehn twist is smoothly trivial rel boundary (\cref{rem: equiv plumbing}).
So the boundary Dehn twist on a 4-manifold bounded by a Seifert 3-manifold may or may not be trivial even among complex surfaces.

\cref{thm intro: one stabilization} is a direct consequence of the following.
Recall that the Milnor fibers $M(2,3,7)$ and $M(2,3,11)$ are naturally embedded in $K3$ (see such as \cite[The sentences after Theorem 8.3.2]{GS99}).
In particular, $\Sigma(2,3,7)$ and $\Sigma(2,3,11)$ are naturally embedded in $K3$. 
We will prove:

\begin{thm}
\label{thm: intro stabilization in K3}
Let $\mathring{K3}$ be a punctured $K3$ surface obtained by removing a $4$-ball in $K3$ that is disjoint from $M(2,3,7)$. Then the Dehn twist along $\Sigma(2,3,7)$ on $\mathring{K3}$ is relatively exotic.
Moreover, if we form the connected sum $\mathring{K3}\# S^2 \times S^2$ by attaching $S^2\times S^2$ away from $\Sigma(2,3,7)$, the Dehn twists along $\Sigma(2,3,7)$ on $\mathring{K3}\# S^2 \times S^2$ is still relatively exotic.
Analogous results hold also for $M(2,3,11)$ 
 and $\Sigma(2,3,11)$ in place of $M(2,3,7)$ 
 and $\Sigma(2,3,7)$.
\end{thm}

We will give another example of an exotic Dehn twist that survives after one stabilization in
\cref{one stabilization 279} that is slightly different from the one above.

\begin{rem}
Our technique using the $2$-parameter families Seiberg--Witten theory over $\RP^2$ can recover a part of the results by Lin~\cite{JL20} that the boundary Dehn twist of the punctured $K3$ is relatively exotic even after one stabilization (\cref{thm Lin}).
\end{rem}

\subsection{Exotic Dehn twists on closed 4-manifolds}

The proof of \cref{main thm1: Dehn twist on definite} indeed yields exotic Dehn twists on closed 4-manifolds:

\begin{thm}
\label{thm: closed Dehn}
For any $n \geq 0$,
there exists a smooth 4-manifold $M$ homeomorphic to
$2\CP^2\#(n+11)(-\CP^2)$ such that $M$ admits a smooth embedding of $\Sigma(2,3,6n+7)$ along which the Dehn twist $t_M : M \to M$ is exotic.

Moreover, if we form the connected sum $M' = M\#k(-\CP^2)$ for $k \geq 1$ away from $\Sigma(2,3,6n+7) \subset M$, 
$t_{M}$ is still exotic on $M'$.
\end{thm}

Here are comparisons of \cref{thm: closed Dehn}  with other known results:
\begin{itemize}
\item \cref{thm: closed Dehn} gives the smallest example of closed 4-manifolds that support exotic diffeomorphisms: the previous smallest known example was due to Ruberman~\cite{Ru98}, $4\CP^2\#21(-\CP^2)$. This was in fact the first example of an exotic diffeomorphism.

\item Before this result, the only known examples of exotic Dehn twists on closed 4-manifolds were the Dehn twist along $S^3$ in the neck of $K3\#K3$ due to Kronheimer and Mrowka~\cite[Teorem~1.1]{KM20}, and the Dehn twist on $K3\#K3\#S^2 \times S^2$ along the neck between two copies of $K3$ due to Lin~\cite{JL20}.

\item \cref{thm: closed Dehn} shows that the Dehn twist on a closed 4-manifold along a general 3-manifold need not be isotopic to the identity after connected sums of $-\CP^2$.
In contrast, it follows from \cite[Theorem 2.4]{Gian08} that
the Dehn twist along embedded $S^3$ turns isotopic to the identity after the connected sum of a single copy of $-\CP^2$.
\end{itemize}

\subsection{Non-extendable loops of diffeomorphisms}
\label{Non-extendable loops of diffeomorphisms}
We will now reformulate our results regarding exotic Dehn twists in terms of loops of diffeomorphisms of Seifert fibered 3-manifolds that topologically extend to some 4-manifolds but not smoothly so (we thank Mark Powell for pointing this out).
This also yields comparison results on $\pi_1$ of automorphism groups.

Let us first make the meaning of the extendability clear.
Recall that, for a general orientable closed 3-manifold $Y$, the inclusion map $\Diff(Y) \hookrightarrow \Homeo(Y)$ is a weak homotopy equivalence (see \cite{Hat80,Ce68,Ha83}).
In particular, we may identify $\pi_i(\Diff(Y))$ with $\pi_i(\Homeo(Y))$ via the induced map $\pi_i(\Diff(Y)) \to \pi_i(\Homeo(Y))$ for any $i$.
For a 4-manifold $X$ bounded by $Y$, we denote by 
$r : \Homeo(X) \to \Homeo(Y)$, $r : \Diff(X) \to \Diff(Y)$,
the restriction maps.
Let us make the following definition:

\begin{defn}
\label{def loop ext}
Let $X$ be a 4-manifold bounded by a 3-manifold $Y$.
If (the homotopy class of) a loop $\gamma \in \pi_1(\Diff(Y)) (\cong \pi_1(\Homeo(Y)))$ satisfies
\begin{align*}
\gamma \in
\im(r_\ast : \pi_1(\Homeo(X)) \to \pi_1(\Homeo(Y))),
\end{align*}
we say that {\it the loop $\gamma$ extends to $X$ topologically}.
If we have
\[
\gamma \notin \im(r_\ast : \pi_1(\Diff(X)) \to \pi_1(\Diff(Y))),
\]
we say that {\it the loop $\gamma$ does not extend to $X$ smoothly}.
\end{defn}

When $Y$ is a Seifert fibered 3-manifold, let 
\[
\gamma_S \in \pi_1(\Diff(Y)) (\cong \pi_1(\Homeo(Y)))
\]
denote the loop of diffeomorphisms arising from the standard circle action on $Y$.
We call $\gamma_S$ the {\it Seifert loop}.
The Seifert loop is non-trivial in $\pi_1(\Diff(Y)) \cong \pi_1(\Homeo(Y))$ \cite[Proposition~8.8]{OrsonPowell2022}, and it extends topologically to any simply-connected 4-manifolds bound by $Y$ \cite{OrsonPowell2022}.
Nevertheless, we will see that the Seifert loop may not extend smoothly:

\begin{thm}
\label{thm: pi1 1}
Let $n \geq 0$ and let $X$ be a compact, oriented, positive-definite, simply-connected, smooth compact 4-manifold bounded by $Y=\Sigma(2,3,6n+7)$.
Then the Seifert loop $\gamma_S \in \pi_1(\Diff(Y))$ extends to $X$ topologically, but not smoothly.
\end{thm} 

Before \cref{thm: pi1 1}, a punctured homotopy $K3$ surface $M$ and $M \# S^2 \times S^2$ were the only known examples of a 4-manifold whose boundary admits a non-extendable loop of diffeomorphisms
\footnote{The results on a punctured homotopy $K3$ and its once stabilization are not written explicitly in the literature, but one can prove these by applying \cref{lem: exact} to the results by Kronheimer--Mrowka~\cite{KM20} and Lin~\cite{JL20} mentioned above.}.
\Cref{thm: pi1 1} gives the first example of such 4-manifolds with boundary $\neq S^3$.

As an immediate consequence of \cref{thm: pi1 1}, we show:

\begin{cor}
\label{thm: pi1 smooth vs. top 1}
Let $n \geq 0$ and $X$ be any compact, oriented, positive-definite, simply-connected, smooth compact 4-manifold bounded by $\Sigma(2,3,6n+7)$.
Then the natural map
\[
\pi_1(\Diff(X)) \to \pi_1(\Homeo(X))
\]
is not surjective.
\end{cor}

The previously known examples of a 4-manifold $X$ and $i>0$ with non-surjective $\pi_i(\Diff(X)) \to \pi_i(\Homeo(X))$ are only when $i=1$, for 4-manifolds : (i) $K3$, (ii) a punctured homotopy $K3$ surface $M$, and (iii) $M\#S^2\times S^2$.
This was proven by Baraglia and the first author \cite{BK21} for (i) and follows from Kronheimer--Mrowka~\cite{KM20} for (ii) and from Lin~\cite{JL20} for (iii) with the help of \cite{OrsonPowell2022}. 
\Cref{thm: pi1 smooth vs. top 1} considerably increases the list of such 4-manifolds $X$, and gives the first example of this phenomenon for 4-manifolds with boundary $\neq S^3$.

We will deduce from \cref{contractible} an analogous result for some contractible 4-manifolds:

\begin{thm}
\label{cor: pi1 non surj contractible}
Let $C$ be a contractible compact smooth 4-manifold bounded by either $\Sigma(2,3,13)$ or  $\Sigma(2,3,25)$.
Then we have:
\begin{itemize}
\item[(i)] The Seifert loop $\gamma_S \in \pi_1(\Diff(\del C))$ extends topologically to $C$ but not smoothly.
\item[(ii)] The natural map
\[
\pi_1(\Diff(C)) \to \pi_1(\Homeo(C))
\]
is not surjective. 
\end{itemize}
\end{thm}
\noindent
Note that there is a formal similarity between non-extendable loops of diffeomorphisms and corks. Recall that, for a cork, a diffeomorphism on the boundary extends topologically over a contractible manifold but not smoothly. Then the assertion (i) of \cref{cor: pi1 non surj contractible} may be viewed as detection of \textit{``$\pi_1$-version of a cork"}.  However, in contrast to our Brieskorn sphere example, it is still not known whether there exists a cork in the classical sense whose boundary is a Brieskorn sphere \cite[Question 3.2]{DHM20}.


A theorem by Saeki \cite{Saeki06} implies that, after sufficiently many stabilizations, the Seifert loop for any Seifert fibered homology 3-sphere extends even smoothly to any simply-connected 4-manifold (\cref{thm Dehn twist stably trivial}).
The results in \cref{subsection intro One stabilization} shall imply that one stabilization is not enough for some Seifert fibered homology 3-sphere:

\begin{thm}
\label{thm pi 1 Milnor necleus}
Let $X$ be either the Milnor fiber $M(2,3,7)$ or $M(2,3,11)$.
Then we have:
\begin{itemize}
\item [(i)] The Seifert loop $\gamma_S \in \pi_1(\Diff(\del X))$ extends to $X$ topologically, but not smoothly, even to $X\# S^2 \times S^2$.
\item [(ii)] Both of the natural maps
\begin{align*}
\pi_1(\Diff(X)) &\to \pi_1(\Homeo(X)),\\
\pi_1(\Diff(X \# S^2 \times S^2)) &\to \pi_1(\Homeo(X \# S^2 \times S^2))
\end{align*}
are not surjective.
\end{itemize}
\end{thm}
\noindent

For most of the results explained until here, the simply-connectedness of 4-manifolds was needed to use the result by Orson and Powell~\cite{OrsonPowell2022} in the topological category, but we may relax the condition $\pi_1=0$ to $b_1=0$ to get the results in the smooth category.
This generalization is somewhat similar to the generalization of the notion of cork called {\it strong cork}, introduced in \cite{LRS18} and studied extensively also in \cite{DHM20,KMT23first}.
For example, we can define the notion of {\it strongly non-extendable loop} of diffeomorphisms of a 3-manifold in a similar spirit (\cref{defn: storngly non-extendable loop}), and we prove that the Seifert loops for $\Sigma(2,3,13)$ and $\Sigma(2,3,25)$ are strongly non-extendable (\cref{strong pi1 cork}).
See \cref{main thm1 generelized: Dehn twist on definite,main thm1 generelized: Dehn twist on definite 2} for this generalization.

\subsection{More exotic diffeomorphisms: exotic commutators}
\label{More exotic diffeomorphisms: exotic commutators}

Using similar ideas as in detecting exotic Dehn twists,
we will also provide a different kind of exotic diffeomorphism.
This demonstrates that the technique we introduce is somewhat flexible and versatile.

To put this new example of an exotic diffeomorphism in a reasonable context, let us consider the following question discussed in \cite{IKMT22} by Mukherjee, Iida, and the first and third authors: For a given closed 3-manifold, is there a smooth compact 4-manifold $X$ bounded by $Y$ such that $X$ admits an exotic diffeomorphism? 
It was proven in \cite{IKMT22} that, if we replace $Y$ with a 3-manifold $Y'$ that is homology cobordant to $Y$, one may solve the above question in the positive for any $Y$ which has  rational homology of either $S^3$ or $S^1 \times S^2$.
However, this problem remains non-trivial if we do not make the above replacement.

Here, we make partial progress in the positive direction for many 3-manifolds obtained as surgeries on knots. To succinctly express our result, we make use of concordance invariant $V_0$ defined by Rasmussen \cite{rasmussen2003floer}, coming from Heegaard Floer homology. 

\begin{thm}
\label{thm: relative commutator}
Let $K$ be a knot in $S^3$ with $V_0(K)>0$.
Then, for each positive integer $n>0$,
$S^3_{1/2n}(K)$ bounds a $4$-manifold $X$ with $b_2(X)=4$ that admits a relatively exotic diffeomorphism $f$.
Similarly, for each $n\geq0$, $S^3_{1/(2n+1)}(K)$ bounds a $4$-manifold $X$ with $b_2(X)=5$ which admits a relative exotic diffeomorphism $f$.

Moreover, these exotic diffeomorphisms $f$ are written as commutators, i.e. $f=[f_1,f_2]$ for some $f_i \in \Diff(X,\del)$.
\end{thm}

\cref{thm: relative commutator} implies that, for example, for the mirror image of any non-slice quasi-positive knot or any quasi-alternating knot with positive signature, the $1/2n$ surgery bounds a smooth 4-manifold with $b_2=4$ which admits a relatively exotic diffeomorphism. 

Before commenting on the last statement of the \cref{thm: relative commutator} on commutators,
let us state an analogous result for closed 4-manifolds: 
\begin{thm}
\label{thm: closed commutator}
Set $X=2\CP^2 \# n(-\CP^2)$ for $n \geq 11$.
Then $X$ admits an exotic diffeomorphism $f$ written as a commutator, i.e. $f=[f_1,f_2]$ for some $f_i \in \Diff(X)$.
\end{thm}

The significance of the existence of {\it exotic commutator diffeomorphism} can be explained by recalling known techniques to detect exotic diffeomorphisms.
Most known exotic diffeomorphisms, obtained in such as \cite{Ru98,Ru99, BK18,IKMT22,KL22}, have been detected by using homomorphisms $G \to A$, from subgroups $G$ of the mapping class group $\pi_0(\Diff(X))$ to some abelian groups $A$, such as $\Z$ or $\Z/2$.
Such homomorphisms descend to the abelianization $\mathrm{Ab}(G) \to A$, which cannot detect the non-triviality of commutators of $G$.
For example, a homomorphism $\pi_0(\Diff(X)) \to \Z/2$ was defined in \cite{KL22}, and using this homomorphism, the authors detected an exotic diffeomorphism $f : X \to X$ that cannot lie in $[\Diff(X), \Diff(X)]$ for some 4-manifold $X$ \cite[Theorem~1.5]{KL22}.
Exotic commutator diffeomorphisms are complementary to such an example.

For both Dehn twists and commutator diffeomorphisms,
all exotic diffeomorphisms detected in this paper are smoothly pseudo-isotopic to the identity, relative to the boundary if $\del X \neq \emptyset$.
See \cref{rem pseudo isotopy}.

\subsection{Scheme of the proof}
\label{Scheme of the proof}

The proofs of the results of this article are based on a strategy that is somewhat different from known arguments to detect exotic diffeomorphisms:  we shall use constraints on families of smooth 4-manifold over $\RP^2$ or $T^2$ from Seiberg--Witten theory.
As mentioned before, all known arguments \cite{Ru98,BK18,KM20,JL20,IKMT22} to detect exotic diffeomorphisms are based on gauge-theoretic invariants of 1-parameter families of 4-manifolds, and this is a reason why it was difficult to detect exotic diffeomorphisms of contractible 4-manifolds and exotic Dehn twists.
Although there are several studies of families gauge theory over higher-dimensional base spaces, those have not been successfully used to detect exotic diffeomorphisms.

We clarify the scheme underlying our proof below.
A key idea is to consider the algebraic structure of the mapping class group.
Let $G = \left<a_1, \ldots, a_n \mid r \right>$ be a finitely presented group with a single relation $r$.
Let $X$ be an oriented smooth 4-manifold, and assume that $X$ is closed for simplicity.
Let $\Diff^+(X)$ denote the orientation-preserving diffeomorphism group, and $\Aut(H_2(X;\Z))$ denote the automorphism group of the intersection lattice.
Given a group homomorphism $\varphi : G \to \Aut(H_2(X;\Z))$,
we may ask whether $\varphi$ lifts to a homomorphism $\tilde{\varphi} : G \to \pi_0(\Diff^+(X))$ along the natural map $\mu : \pi_0(\Diff^+(X)) \to \Aut(H_2(X;\Z))$:
\begin{align*}
\xymatrix{
     & \pi_0(\Diff^+(X)) \ar[d]^{\mu}\\
    G \ar@{-->}[ru]^-{\tilde{\varphi}} \ar[r]_-{\varphi}  &  \Aut(H_2(X;\Z)).
    }
\end{align*}

For certain $G$, $X$, and $\varphi$, we will show that there is \textit{no} such lift $\tilde{\varphi}$ using families Seiberg--Witten theory developed by Baraglia~\cite{Ba21} and the first and third authors~\cite{KT22} applied to a family of $X$ over the 2-skeleton $(BG)^{(2)}$ of the classifying space, in other words, the presentation complex of $G$.
Let $r = r(a_1, \ldots, a_n)$ be the expression of $r$ as a word of $a_1, \ldots, a_n$.
If the image of $\varphi$ is contained in the image of $\mu$, we may find diffeomorphisms $f_1, \ldots, f_n \in \Diff^+(X)$ with $\mu([f_i]) = \varphi(a_i)$.
Set $F = r(f_1, \ldots, f_n) \in \Diff^+(X)$.
Then the non-liftability of $\varphi$ implies that $[F] \neq 1$ in $\pi_0(\Diff^+(X))$, while $F_\ast=1$ in $\Aut(H_2(X;\Z))$.
Since $\Aut(H_2(X;\Z))$ is isomorphic to $\pi_0(\Homeo(X))$ as far as $X$ is simply-connected \cite{Q86,P86}, it follows that $F$ is an exotic diffeomorphism.

We shall carry out the above idea for $G=\Z/2 = \left<a \mid a^2\right>$ and $G=\Z^2 = \left<a_1,a_1 \mid [a_1,a_2]\right>$.
For example, for $G=\Z/2=\left<a \mid a^2\right>$, we define $\varphi(a)$ as the action of a `square root' of a Dehn twist on the intersection form.
Then we solve the above lifting problem in the negative using constraints from Seiberg--Witten theory for families over $(B(\Z/2))^{(2)}=\RP^2$ and $(B\Z^2)^{(2)}=T^2$, which yield exotic Dehn twists and exotic commutator diffeomorphisms respectively.

\subsection{Structure of the paper}

In \cref{Input from family Seiberg--Witten theory}, we collect necessary results from families Seiberg--Witten theory.
In \cref{section Families over RP2}, we give constraints on diffeomorphisms using results in \cref{Input from family Seiberg--Witten theory} applied to families over $\RP^2$, which are crucial to the results on exotic Dehn twists in this paper.
Similarly, in \cref{section Families over T2}, we consider families over $T^2$ which shall be used to get the results on commutator exotic diffeomorphism.
In \cref{section Proof of reluts on Dehn twists}, we give the proofs of the results on Dehn twists by combining the result in \cref{section Families over RP2} with constructions of equivariant cobordisms.
In \cref{section Loops of diffeomorphisms}, we translate the results in \cref{section Proof of reluts on Dehn twists} in terms of loops of diffeomorphisms to get the results in \cref{Non-extendable loops of diffeomorphisms}.
In \cref{section Proof of reluts on commutators}, we give the proofs of the results on exotic commutators (\cref{More exotic diffeomorphisms: exotic commutators}).

\begin{acknowledgement}
We would like to thank Patrick Orson and Mark Powell for attracting our attention to the problem discussed in the current paper. Also, we wish to thank Mark Powell for pointing out that exotica of Dehn twists can be rephrased in terms of non-extendable loops of diffeomorphisms as in \cref{Non-extendable loops of diffeomorphisms,section Loops of diffeomorphisms}.
We learned a lot also from Kronheimer's talk \cite{Krontalk} at Gauge theory virtual, so we are grateful to Kronheimer, as well as the organizers of Gauge theory virtual.
We would also like to thank Danny Ruberman, Ian Montague, Charles R. Stine, Tadayuki Watanabe, Kristen Hendricks and Sally Collins for  helpful conversations, and thank David Baraglia for comments on a draft of the paper. The authors thank the anonymous referee for providing such a detailed and insightful report which helped to improve this article.
This project began in the program entitled `Floer homotopy theory' held at  MSRI/SL-Math on Aug-Dec in 2022, supported by the NSF under Grant No. DMS-1928930. The authors thank the organizers of the program for inviting them. In addition, HK was partially supported by JSPS KAKENHI Grant Numbers 17H06461, 19K23412, and 21K13785, AM was partially supported by NSF DMS-2019396. MT was partially supported by JSPS KAKENHI Grant Number 20K22319, 22K13921, and RIKEN iTHEMS Program.
\end{acknowledgement}

\section{Input from families Seiberg--Witten theory}
\label{Input from family Seiberg--Witten theory}

We summarize results in families Seiberg--Witten theory which is crucial in the proof of the results in this paper.
First, we recall some basics of families of 4-manifolds.
See \cite[Subsection~2.3]{KT22} for more details.

Let $X$ be an oriented smooth 4-manifold.
Recall that the space of maximal-dimensional positive-definite subspaces of $H^2(X;\R)$ is contractible.
We often write just $H^+(X)$ for a choice of such subspace.
Let $\Diff^+(X)$ denote the orientation-preserving diffeomorphism group.
Let $X \to E \to B$ be a smooth fiber bundle with fiber $X$ over a CW complex $B$.
Then we have an associated vector bundle $\R^{b^+(X)} \to H^+(E) \to B$ of rank $b^+(X)$.
This vector bundle is not determined by $E$, but the isomorphism class of $H^+(E)$ is determined only by $E$ because the space of choices of $H^+(X)$ is contractible.
Intuitively, $H^+(E)$ is obtained by collecting $H^+(E_b)$ over the fibers $E_b$ of $E$.

Let $\fraks$ be a spin$^c$ structure on $X$, and
let $\Aut(X,\fraks)$ denote the automorphism group of the spin$^c$ 4-manifold $(X,\fraks)$, defined without using Riemannian metric.
Then the kernel of the natural forgetful map $\Aut(X,\fraks) \to \Diff(X)$ is given by the gauge group $\Map(X,U(1))$.

Similarly, if $X$ has non-empty boundary, a group $\Aut((X,\fraks),\del)$ is defined to be the structure group of families of spin$^c$ 4-manifolds with fiber $X$ whose restriction to the boundary are trivialized as a family of spin$^c$ 3-manifolds.

Here we note a sufficient condition to reduce the structure group of a smooth fiber bundle to the automorphism group of a spin$^c$ structure.
Recall that, in general, given a smooth fiber bundle $X \to E \to B$ with structure group $\Diff^+(X)$, it gives rise to the monodromy $\pi_1(B) \to \pi_0(\Diff^+(X))$, and similarly for $\Diff(X,\del)$.

\begin{lem}[{\it cf. }{\cite[Proposition~2.1]{B19}}]
\label{lem: Fiberwise spinc}
Let $X$ be a compact oriented smooth 4-manifold with $b_1(X)=0$.
Let $\fraks$ be a spin$^c$ structure on $X$.
Let $X \to E \to B$ be a smooth fiber bundle over a CW complex $B$ with $H^3(B;\Z)=0$.
Suppose that $\fraks$ is preserved under the monodromy action of $E$.
Furthermore, we suppose the following:
\begin{itemize}
\item If $X$ has empty boundary, then we suppose that the structure group of $E$ reduces to $\Diff^+(X)$.
\item If $X$ has non-empty boundary, we suppose that $\del X=Y$ is an integral homology 3-sphere and the structure group of $E$ reduces to $\Diff(X,\del)$.
\end{itemize}
Then the structure group of $E$ reduces to $\Aut(X,\fraks)$ or $\Aut((X,\fraks),\del)$, according to $\del X = \emptyset$ or $\del X \neq \emptyset$.
\end{lem}

\begin{proof}
The case for $\del X = \emptyset$ has been proven by Baraglia~\cite[Proposition~2.1]{B19}.
We can extend his result to the case for $\del X \neq \emptyset$ as follows.

First, just by repeating the proof of \cite[Proposition~2.1]{B19}, one can see that the assumption $H^3(B;\Z)=0$ implies that the structure group of $E$ reduces to $\Aut(X,\fraks)$ even when $\del X \neq \emptyset$.
Namely, we have a spin$^c$ structure $\fraks_E$ on the vertical tangent bundle $T_v E \to E$ that restricts to $\fraks$ on the fibers.
We may regard $\fraks_E$ as a family of spin$^c$ structures with typical fiber $\fraks$, but $\fraks_E$ is not necessarily a trivial family of spin$^c$ structures on the family of the boundaries $E_{\del} = Y \times B$.
We will modify $\fraks_E$ so that the restriction to $E_{\del}$ will be the trivial family of spin$^c$ structures.

Let $\Spin^c(T_vE)$ denote the set of spin$^c$ structures of the vector bundle $T_v E$.
Recall that $\Spin^c(T_vE)$  is acted by $H^2(E;\Z)$ freely and transitively.
We identify $H^2(E;\Z)$ with the set of isomorphism classes of complex line bundles over $E$, and we denote the action of $\mathcal{L} \in H^2(E;\Z)$ by $-\otimes \mathcal{L}$.
Let $\frakt$ denote the spin$^c$ structure on $T_v E_\del$ defined as the pull-back of $\fraks|_Y$ under the projection 
\[
T_v E_\del = TY \times B \to TY.
\]
Then there is a line bundle $L \in H^2(E_\del;\Z)$ such that $\frakt = \fraks_E|_{E_\del} \otimes L$.
Since we supposed that $Y$ is an integral homology 3-sphere, the projection $E_\del \to B$ induces an isomorphism $H^2(E_\del;\Z) \cong H^2(B;\Z)$.
Let $L_0 \in H^2(B;\Z)$ denote the cohomology class corresponding to $L$ via this isomorphism.

Let $\pi : E \to B$ denote the projection.
Define $\tilde{L} \in H^2(E;\Z)$ by $\tilde{L} = \pi^\ast L_0$.
By construction, we have $\tilde{L}|_{E_\del} = L$.
Furthermore, $\tilde{L}$ restricts to the trivial line bundle on each fiber $E_b$ of $E \to B$ for $b \in B$.
Hence, if we put $\fraks_E' = \fraks_E \otimes \tilde{L}$, this new spin$^c$ structure $\fraks_E'$ on $T_v E$ restricts to $\fraks$ on each fiber of $E \to B$, and $\fraks_E'$ restricts to $\frakt$ on $E_\del$.
Therefore $\fraks_E'$ gives a lift of the structure group of $E$ from $\Diff(X,\del)$ to $\Aut((X,\fraks),\del)$.
This completes the proof.
\end{proof}

\begin{rem}
In the proof of \cref{lem: Fiberwise spinc}, the assumption that $Y$ is an integral homology 3-sphere was used only to ensure that there is a line bundle $L_0 \in H^2(B;\Z)$ such that $L = \pi_\del^\ast L_0$, where $\pi_\del : E_\del \to B$ is the projection.
So we can drop the homology sphere assumption as far as we know the existence of such $L_0$.
This may happen often, as $L \to E_{\del}$ restricts to the trivial line bundle on each fiber $Y \times \{b\}$ for $b \in B$, since both of $\frakt$ and $\fraks_E|_{E_\del}$ restrict to a common spin$^c$ structure, $\fraks|_Y$, on the fibers of $T_v E_\del \to B$.
\end{rem}

The first key inputs in the proofs of the results in this paper are the following family version of Donaldson's diagonalization due to Baraglia~\cite{Ba21} and family version of Fr{\o}yshov's inequality by the first and third authors \cite{KT22}:

\begin{thm}[Baraglia~{\cite[Theorem~1.1]{Ba21}}]
\label{thm: Baraglia}
Let $X$ be a closed oriented smooth 4-manifold.
Let $\fraks$ be a spin$^c$ structure on $X$.
Let $X \to E \to B$ be a fiber bundle with structure group $\Aut(X,\fraks)$ over  a compact base space $B$.
Suppose that $w_{b^+(X)}(H^+(E)) \neq 0$.
Then we have 
\[
c_1(\fraks)^2 - \sigma(X)
\leq 0.
\]
\end{thm}

\begin{thm}[{\cite[Theorem~1.1]{KT22}}]
\label{thm: KT}
Let $X$ be a compact oriented smooth 4-manifold with $b_1(X)=0$ and with rational homology 3-sphere boundary $Y$.
Let $\fraks$ be a spin$^c$ structure on $X$.
Let $X \to E \to B$ be a fiber bundle with structure group $\Aut((X,\fraks),\del)$ over  a compact base space $B$.
Suppose that $w_{b^+(X)}(H^+(E)) \neq 0$.
Then we have 
\[
\frac{c_1(\fraks)^2 - \sigma(X)}{8}
\leq \delta(Y,\fraks|_{Y}).
\]
\end{thm}

Here $\delta(-)$ denotes the (Seiberg--Witten) Fr{\o}yshov invariant.
We follow the notation and convention of \cite{Ma16}.

\begin{rem}
As noted in \cite[Section~7]{Ba21}, one does not have to assume that $b_1(X)=0$ in \cref{thm: Baraglia}.
\end{rem}

\begin{rem}
In the proof of \cref{main thm1: Dehn twist on definite}, it is enough to use \cref{thm: Baraglia} rather than \cref{thm: KT}, while we need \cref{thm: KT} to prove other results in this paper.
\end{rem}

Next, we will recall an analogous result for spin structures instead of spin$^c$ structures, which shall be crucial to proving the results on the stabilization by $S^2 \times S^2$ (\cref{subsection intro One stabilization}).
Let $\Aut(X,\fraks)$ denote the automorphism group of a spin 4-manifold $(X,\fraks)$, just as in the spin$^c$ case by abuse of notation.
The kernel of the forgetful map $\Aut(X,\fraks) \to \Diff(X)$ is $\Z/2$.
If $X$ has non-empty boundary, a group $\Aut((X,\fraks),\del)$ is defined to be the structure group of families of spin 4-manifolds with fiber $X$ whose restriction to the boundary are trivialized as a family of spin 3-manifolds.
(See \cite[Subsection~2.3]{KT22} for details.)
The following is a families and relative version of Donaldson's Theorems B, C.
This is a relative version of a result by Baraglia~\cite[Theorem~1.2]{Ba21} for closed spin 4-manifolds.
However, we will see that families of {\it closed} spin 4-manifolds do not work for our purpose (\cref{rem: why not closed for spin}), so let us state only the relative version:

\begin{thm}[{\cite[Theorem~1.2]{KT22}}]
\label{thm: KT spin}
Let $X$ be a compact oriented smooth 4-manifold with $b_1(X)=0$ and with rational homology 3-sphere boundary $Y$.
Let $\fraks$ be a spin structure on $X$.
Let $X \to E \to B$ be a fiber bundle with structure group $\Aut((X,\fraks),\del)$ over  a compact base space $B$.
Then:
\begin{enumerate}
\item [(i)] If $w_{b^+(X)}(H^+(E)) \neq 0$, then we have 
\[
-\frac{\sigma(X)}{8}
\leq \gamma(Y,\fraks|_{Y}).
\]
\item [(ii)] If $w_{b^+(X)-1}(H^+(E)) \neq 0$, then we have 
\[
-\frac{\sigma(X)}{8}
\leq \beta(Y,\fraks|_{Y}).
\]
\item [(iii)] If $w_{b^+(X)-2}(H^+(E)) \neq 0$, then we have 
\[
-\frac{\sigma(X)}{8}
\leq \alpha(Y,\fraks|_{Y}).
\]
\end{enumerate}
\end{thm}

Here $\alpha, \beta, \gamma$ are Manolescu's invariants defined in \cite{Ma16}.

\section{Families over $\RP^2$}
\label{section Families over RP2}

A key idea to get results on Dehn twists is to consider families over $\RP^2$.
We will consider both families of spin$^c$ 4-manifolds and families of spin 4-manifolds in \cref{subsection Families of spinc 4-manifolds} and \cref{subsection Families of spin 4-manifolds} respectively.

\subsection{Families of spin$^c$ 4-manifolds}
\label{subsection Families of spinc 4-manifolds}

Given an oriented compact 4-manifold $X$, as mentioned above, there are many choices of maximal-dimensional positive definite subspaces $H^+(X) \subset H^2(X;\R)$.
Given a subgroup $G$ of the automorphism group $\Aut(H_2(X;\Z))$ of the intersection form, if there is a choice $H^+(X) \subset H^2(X;\R)$ which is preserved setwise under the $G$-action, we just say that $G$ acts on $H^+(X)$.
First, we shall prove:

\begin{thm}
\label{thm: rel RP2 spinc pres}
Let $X$ be a compact, oriented, smooth 4-manifold with boundary and with $b^+(X)=2$ and $b_1(X)=0$.
Set $Y = \del X$ and suppose that $Y$ is an integral homology 3-sphere.
Let $\fraks$ be a spin$^c$ structure on $X$.
Suppose that there is a diffeomorphism $f \in \Diff(X,\del)$ such that $f^2$ is isotopic to the identity through $\Diff(X,\del)$, and $f^\ast$ acts on $H^+(X)$ as via multiplication by $-1$, with $f^\ast \fraks \cong \fraks$.
Then  we have
\begin{align}
\label{eq: Fro ineq}
\frac{c_1(\fraks)^2-\sigma(X)}{8} \leq \delta(Y,\fraks|_Y).
\end{align}
\end{thm}

The proof of \cref{thm: rel RP2 spinc pres} is based on the following elementary observation:

\begin{lem}
\label{lem: RP2 family constrcution}
Let $X$ be a smooth manifold with boundary.
Let $f \in \Diff(X,\del)$ be a diffeomorphism such that $f^2$ is isotopic to the identity through $\Diff(X,\del)$.
Then there exists a fiber bundle $X \to E \to \RP^2$ with structure group $\Diff(X,\del)$ whose monodromy is given by $[f] \in \pi_0(\Diff(X,\del))$.
\end{lem}

\begin{proof}
Let us consider the standard cellular decomposition
\[
\RP^2 = e^0 \cup e^1 \cup e^2
\]
of $\RP^2$, where $e^i$ are $i$-dimensional cells.
Let $\varphi : S^1 \to e^0 \cup e^1 =\RP^1$ be the attaching map for $e^2$, which is given as the double covering map $S^1 \to S^1$.
In general, for a diffeomorphism $F$ of $X$, let $X \to X_F \to S^1$ denote the mapping torus of $F$.
Define a fiber bundle $X \to E^{(1)} \to \RP^1=S^1$ by $E^{(1)}=X_f$.
We extend $E^{(1)}$ to the whole $\RP^2$ as follows.

First, the pull-back of the monodromy of $E^{(1)}$ under $\varphi$ is $f^2$, which is, by our assumption, isotopic to the identity through an isotopy $\{F_t\}_{t \in [0,1]}$ in  $\Diff(X,\del)$.
We can form a 1-parameter family of mapping torus $X_{F_t} \to S^1$, and regard it as a fiber bundle $X \to E' \to [0,1] \times S^1$.
Then $E'|_{\{0\} \times S^1}$ is $\varphi^\ast E^{(1)}$, and $E'|_{\{1\} \times S^1}$ is the trivial bundle.
Glue $E'$ with the trivial bundle $D^2 \times X \to D^2$ along $\{1\} \times S^1$ and get a family over $([0,1] \times S^1) \cup D^2 \cong D^2$, which is regarded as the domain of the characteristic map of $e^2$.
Thus we can extend $E^{(1)}$ to the whole $\RP^2$.
This extension $X \to E \to \RP^2$ gives the desired fiber bundle.
\end{proof}

\begin{proof}[Proof of \cref{thm: rel RP2 spinc pres}]
Let $X \to E \to \RP^2$ be the fiber bundle given by \cref{lem: RP2 family constrcution}.
By \cref{lem: Fiberwise spinc}, the structure group of $E$ reduces to $\Aut((X,\fraks),\del)$.

Let $\gamma_{\R} \to \RP^2$ denote the tautological real line bundle.
We claim that the plane bundle 
$H^+(E) \to \RP^2$ is given by $\gamma_{\R} \oplus \gamma_{\R} \to \RP^2$.
To see this, let $\tilde{E} \to S^2$ denote the pull-back of $E \to \RP^2$ under the natural surjection $S^2 \to \RP^2$.
Since the family $H^2(X;\R) \to H(\tilde{E}) \to S^2$ associated with $\tilde{E}$ is a flat bundle and $\pi_1(S^2)=1$, we have a canonical trivialization $H(\tilde{E}) \cong H^2(X;\R) \times S^2$.
Under this trivialization, the family $H^2(X;\R) \to H(E) \to \RP^2$ associated with $E$ is given by $H(E) \cong H^2(X;\R) \times_{\Z/2} S^2$.
Here $\Z/2$ acts on $S^2$ as the antipodal map, and acts on $H^2(X;\R)$ via $f^\ast$.
Using the same trivialization, the $f^\ast$-invariant subspace $H^+(X)$ gives rise to the family $H^+(E) \cong H^+(X)\times_{\Z/2}S^2 \to \RP^2$.
Since we assumed that $f$ acts on $H^+(X)$ via multiplication by $-1$, we have that $H^+(E)\cong \gamma_{\R} \oplus \gamma_{\R} \to \RP^2$.

Thus we have 
\[
w_2(H^+(E))
= w_2(\gamma_{\R} \oplus \gamma_{\R})
= w_1(\gamma_{\R})^2 \neq 0.
\]
Hence the desired inequality \eqref{eq: Fro ineq} follows from \cref{thm: KT}.
\end{proof}

\begin{rem}
In the proof of \cref{lem: RP2 family constrcution}, the fiber bundle $E$ is not necessarily constructed from $f$ in a unique way, since it involves the  choice of isotopy $F_t$.
A different choice $F_t'$ of isotopy together with $F_t$ forms a loop in $\Diff(X,\del)$.
If this loop is non-trivial in $\pi_1(\Diff(X,\del))$, the family over $\RP^2$ constructed from $F_t'$ may not be isomorphic to $E$.
\end{rem}

As expected, a statement analogous to \cref{thm: rel RP2 spinc pres} holds for closed 4-manifolds:

\begin{thm}
\label{thm: RP2 spinc pres}
Let $X$ be a closed, oriented, smooth 4-manifold with $b^+(X)=2$ and $b_1(X)=0$.
Let $\fraks$ be a spin$^c$ structure on $X$.
Suppose that there is a diffeomorphism $f \in \Diff^+(X)$ such that $f^2$ is isotopic to the identity, and $f^\ast$ acts on $H^+(X)$ via multiplication by $-1$, and $f^\ast \fraks \cong \fraks$.
Then  we have
\begin{align}
\label{eq: Fro closed}
c_1(\fraks)^2-\sigma(X) \leq 0.
\end{align}
\end{thm}

\begin{proof}
We can repeat the proof of  \cref{thm: rel RP2 spinc pres} using \cref{thm: Baraglia} in place of \cref{thm: KT}.
\end{proof}

\subsection{Families of spin 4-manifolds}
\label{subsection Families of spin 4-manifolds}

Next, we will consider families of spin 4-manifolds over $\RP^2$.
Let us define the following notation.
Let $f : X \to X$ be a diffeomorphism of a 4-manifold $X$ such that $f^2$ acts on homology trivially.
Then $f$ gives rise to an involution on $H^2(X;\R)$.
Suppose that we may choose $H^+(X) \subset H^2(X;\R)$ so that $H^+(X)$ is preserved setwise under the involution $f^\ast$.
Let $H^+(X)^f$ denote the $f^\ast$-invariant part of $H^+(X)$.

\begin{thm}
\label{thm: rel RP2 spin pres}
Let $X$ be a compact, oriented, smooth 4-manifold with boundary and with $2 \leq b^+(X) \leq 4$ and $b_1(X)=0$.
Set $Y = \del X$ and suppose that $Y$ is a rational homology 3-sphere.
Let $\fraks$ be a spin structure on $X$.
Suppose that there is a diffeomorphism $f \in \Diff(X,\del)$ such that $f^2$ is isotopic to the identity through $\Diff(X,\del)$ and $f^\ast \fraks \cong \fraks$.
Suppose also that
\[
2 \leq b^+(X) - \dim H^+(X)^f \leq 3.
\]
Then:
\begin{enumerate}
\item [(i)] If $b^+(X)=2$, then we have
\begin{align*}
-\frac{\sigma(X)}{8} \leq \gamma(Y,\fraks|_Y).
\end{align*}
\item [(ii)] If $b^+(X)=3$, then we have
\begin{align*}
-\frac{\sigma(X)}{8} \leq \beta(Y,\fraks|_Y).
\end{align*}
\item [(iii)] If $b^+(X)=4$, then we have
\begin{align*}
-\frac{\sigma(X)}{8} \leq \alpha(Y,\fraks|_Y).
\end{align*}
\end{enumerate}
\end{thm}

The proof of \cref{thm: rel RP2 spin pres} is based on the following variant of \cref{lem: RP2 family constrcution}:

\begin{lem}
\label{lem: RP2 family constrcution spin}
Let $X$ be a smooth manifold with boundary.
Let $f \in \Diff(X,\del)$ be a diffeomorphism such that $f^2$ is isotopic to the identity through $\Diff(X,\del)$.
Let $\fraks$ be a spin structure on $X$ and suppose that $f^\ast \fraks \cong \fraks$.
Then there exists a fiber bundle $X \to E \to \RP^2$ with structure group $\Aut((X,\fraks),\del)$ whose monodromy reduced to $\pi_0(\Diff(X,\del))$ is given by $[f] \in \pi_0(\Diff(X,\del))$.
\end{lem}

\begin{proof}
Set $Y=\del X$ and $\frakt = \fraks|_{\del X}$.
Let $\Diff^+(X,[\fraks])$ denote the group of orientation-preserving diffeomorphisms that preserve the isomorphism class $[\fraks]$ of the spin structure $\fraks$.
Then we have a covering fibration
\begin{align}
\label{eq: covering auto}
\Z/2 \to \Aut(X,\fraks) \xrightarrow{\pi} \Diff^+(X,[\fraks]), 
\end{align}
and similarly, we have 
\[
\Z/2 \to \Aut(Y,\frakt) \xrightarrow{\pi} \Diff^+(Y,[\frakt]).
\]

Let $\tilde{f} \in \Aut(X,\fraks)$ be a lift of $f$ along $\pi$; $\pi(\tilde{f})=f$.
Since $\tilde{f}|_{Y}$ covers $f|_{Y} = \id_{Y}$, we have $\tilde{f}|_{Y} \in \{\id_{(Y,\frakt)}, -\id_{(Y,\frakt)} \}$.
By replacing $\tilde{f}$ with $-\tilde{f}$ if necessary, we can assume that $\tilde{f}|_{Y} = \id_{(Y,\frakt)}$ so that $\tilde{f} \in \Aut((X,\fraks),\del)$.
By our assumption, there is an isotopy $\{F_t\}_{t \in [0,1]} \subset \Diff(X,\del)$ from $f^2$ to $\id_X$.
By the homotopy lifting property, there is a 1-parameter family $\{\tilde{F}_t\}_{t \in [0,1]} \subset \Aut(X,\fraks)$ such that $\tilde{F}_0=\tilde{f}^2$ and $\pi(\tilde{F}_t)=F_t$ for all $t$.

The property $\pi(\tilde{F}_t)=F_t$ for $t=1$ implies that $\tilde{F}_1 = \id_{(X,\fraks)}$ or $\tilde{F}_1 = -\id_{(X,\fraks)}$.
We claim that $\tilde{F}_1 = \id_{(X,\fraks)}$.
Indeed, since all $F_t$ lie in $\Diff(X,\del)$, we have $F_t|_{Y} \equiv \id_{Y}$, and it follows from this that $\tilde{F}_t|_{Y} \equiv \id_{(Y,\frakt)}$ or $\tilde{F}_t|_{Y} \equiv -\id_{(Y,\frakt)}$.
Since we know that $\tilde{F}_0 = \tilde{f}^2|_{Y} = \id_{(Y,\frakt)}$, we obtain $\tilde{F}_t|_{Y} \equiv \id_{(Y,\frakt)}$, in particular $\tilde{F}_1|_{Y} = \id_{(Y,\frakt)}$.
Combining this with that $\tilde{F}_1 = \id_{(X,\fraks)}$ or $\tilde{F}_1 = -\id_{(X,\fraks)}$, we can deduce that $\tilde{F}_1 = \id_{(X,\fraks)}$.

Therefore, we have an isotopy $\tilde{F}_t$ from $\tilde{f}^2$ to $\id_{(X,\fraks)}$.
By repeating the proof of 
\cref{lem: RP2 family constrcution} using $\tilde{f}$ and $\tilde{F}_t$ in place of $f$ and $F_t$, we obtain a family of spin 4-manifolds $(X,\fraks) \to E \to \RP^2$ whose monodromy is given by $[\tilde{f}] \in \pi_0(\Aut((X,\fraks),\del))$, which maps to $[f] \in \pi_0(\Diff(X,\del))$ under the forgetful map $\pi_0(\Aut((X,\fraks),\del)) \to \pi_0(\Diff(X,\del))$.
Since we have seen that $\tilde{F}_t|_{Y} \equiv \id_{(Y,\frakt)}$, the structure group of $E$ reduces to $\Aut((X,\fraks),\del)$.
This completes the proof of the \lcnamecref{lem: RP2 family constrcution spin}.
\end{proof}

\begin{proof}[Proof of \cref{thm: rel RP2 spin pres}]
Let $(X,\fraks) \to E \to \RP^2$
be the family of spin 4-manifolds with fiber $(X,\fraks)$ with structure group $\Aut((X,\fraks),\del)$ given by \cref{lem: RP2 family constrcution spin}.
By \cref{thm: KT spin}, it suffices to prove that $w_2(H^+(E)) \neq 0$.

First, let us consider the case that $b^+(X) - \dim H^+(X)^f=2$.
Let $\gamma_\R \to \RP^2$ denote the non-trivial real line bundle and let $\underline{\R} \to \RP^2$ denote the trivial real line bundle.
Then, as in the proof of \cref{thm: rel RP2 spinc pres}, we have $H^+(E)$ is isomorphic to
$\gamma_{\R}^{\oplus 2} \oplus \underline{\R}^{b^+(X)-2}$.
Hence the total Stiefel--Whitney class of $H^+(E)$ is:
\[
w(H^+(E))
= w(\gamma_{\R}^{\oplus 2})
= (1 + w_1(\gamma_{\R}))^2
= 1 + w_1(\gamma_{\R})^2.
\]
Thus we have $w_2(H^+(E)) = w_1(\gamma_{\R})^2 \neq 0$.

Next, let us consider the case that $b^+(X) - \dim H^+(X)^f = 3$.
In this case, $H^+(E)$ is isomorphic to $\gamma_{\R}^{\oplus 3} \oplus \underline{\R}^{b^+(X)-3}$.
Thus we have 
\[
w(H^+(E))
= w(\gamma_{\R}^{\oplus 3})
= (1 + w_1(\gamma_{\R}))^3
= 1 + w_1(\gamma_{\R}) + w_1(\gamma_{\R})^2 + w_1(\gamma_{\R})^3,
\]
and hence $w_2(H^+(E)) = w_1(\gamma_{\R})^2 \neq 0$.
This completes the proof.
\end{proof}

We make a few remarks on the cases which are not treated in \cref{thm: rel RP2 spin pres}:

\begin{rem}
\label{rem: why not closed for spin}
One cannot have an ``absolute" version of \cref{thm: rel RP2 spin pres} for closed spin 4-manifolds, obtained by replacing $\Diff(X,\del)$ and    $\Aut((X,\fraks), \del)$ with $\Diff(X)$ and $\Aut(X,\fraks)$.
Indeed, there is a smooth involution $\iota : K3 \to K3$ that satisfies $b^+(K3) - \dim H^+(K3)^\iota = 2$, but as $-\sigma(K3)>0$, $\iota$ gives a counterexample to the ``absolute" version of the statement (ii) of \cref{thm: rel RP2 spin pres}.
(For instance, one can find such $\iota$ as the free covering involution of $K3$ whose quotient yields the Enriques surface.)
The proof of \cref{thm: rel RP2 spin pres} fails for $\iota$ because there is no lift $\tilde{\iota}$ of $\iota$ to the spin structure which is of order 2; actually, every lift of such $\iota$ is of order 4 by a theorem due to Bryan \cite[Theorem~1.8]{Br98}.
This also gives a counterexample of a ``spin version" of \cref{lem: Fiberwise spinc}.
\end{rem}

\begin{rem}
One can easily see $w_2(\gamma_{\R}^{\oplus 4}) = 0$, so the proof of \cref{thm: rel RP2 spin pres} does not work for the case that $b^+(X) - \dim H^+(X)^f=4$. 
\end{rem}

\begin{rem}
For $b^+(X) - \dim H^+(X)^f=1$, a statement analogous to \cref{thm: rel RP2 spin pres} does not hold.
For example, let $N(2)$ denote the Gompf nucleus with $\del N(2) = -\Sigma(2,3,11)$.
Then it is easy to find an orientation-preserving diffeomorphism $f$ of $X=N(2) \# S^2 \times S^2$ such that $f|_{\del X} =\id_{\del X}$ and $b^+(X) - \dim H^+(X)^f=1$.
However we have $-\sigma(X)/8 = 0 > -2 = \gamma(\del X)$, so a statement analogous to (i) of \cref{thm: rel RP2 spin pres} is violated.
\end{rem}

\section{Families over $T^2$}
\label{section Families over T2}

In this \lcnamecref{section Families over T2}, we shall discuss an analogy of \cref{section Families over RP2} for $T^2$ instead of $\RP^2$.
Let $\diag(a_1, \ldots, a_n)$ denote the diagonal matrix with diagonal entries $a_1, \ldots, a_n$.
The main aim is to prove:

\begin{thm}
\label{thm: rel T2 spinc pres}
Let $X$ be a compact, oriented, smooth 4-manifold with boundary and with $b^+(X)=2$ and $b_1(X)=0$.
Set $Y = \del X$ and suppose that $Y$ is a rational homology 3-sphere.
Let $\fraks$ be a spin$^c$ structure on $X$.
Suppose that there are diffeomorphisms $f_1, f_2 \in \Diff(X,\del)$ such that $[f_1, f_2]$ is isotopic to the identity through $\Diff(X,\del)$, 
and $f_1^\ast, f_2^\ast$ act on $H^+(X)$ as $\diag(-1,1), \diag(1,-1)$ for some choice of basis of $H^+(X)$, and $f_i^\ast \fraks \cong \fraks$.
Then  we have
\[
\frac{c_1(\fraks)^2-\sigma(X)}{8} \leq \delta(Y,\fraks|_Y).
\]
\end{thm}

As in the proof of \cref{thm: rel RP2 spinc pres}, the proof of \cref{thm: rel T2 spinc pres} is based on the following elementary observation, which has appeared in \cite[Section~4]{BK21}, and we follow the argument there.

\begin{lem}
\label{lem: T2 family constrcution}
Let $X$ be a smooth manifold with boundary.
Let $f_1, f_2 \in \Diff(X,\del)$ be diffeomorphisms such that $[f_1,f_2]$ is isotopic to the identity through $\Diff(X,\del)$.
Then there exists a fiber bundle $X \to E \to T^2$ with structure group $\Diff(X,\del)$ whose monodromy is given by $\Z^2 \to \pi_0(\Diff(X,\del))\ ;\ (1,0) \mapsto [f_1],\ (0,1) \mapsto [f_2]$.
\end{lem}

\begin{proof}
Consider the standard cellular decomposition
\[
T^2 = e^0 \cup e_1^1 \cup e_2^1 \cup e^2.
\]
of $T^2$ with two 1-cells $e_i^1$ and a single 0-cell $e^0$ and single 2-cell $e^2$.

Regard $(T^2)^{(1)}$ as the wedge sum of two copies of $S^1$. 
Consider the mapping tori of $f_1$ and $f_2$, and glue them to form a fiber bundle $E^{(1)}$ over $(T^2)^{(1)}$ with structure group $\Diff(X,\del)$.

Our task in to extend $E^{(1)}$ to the whole $T^2$.
As in the proof of \cref{lem: RP2 family constrcution}, the monodromy along the boundary of $D^2$ corresponding to $e^2$ is given by $[f_1,f_2]$.
Regard $D^2$ as a union of a copy of $D^2$ and $[0,1] \times S^1$.
Consider the mapping torus of $[f_1,f_2]$ over $\{0\} \times S^1$.
As we assumed that $[f_1, f_2]$ is isotopic to the identity through $\Diff(X,\del)$,
using the parameter $[0,1]$, we may extend this mapping torus over $D^2 \cong ([0,1] \times S^1) \cup D^2$ by capping with the trivial bundle over $D^2$.
Thus we may extend $E^{(1)}$ to $T^2$ and the resulting fiber bundle $X \to E \to T^2$ with structure group $\Diff(X,\del)$ is the desired one.
\end{proof}

\begin{proof}[Proof of \cref{thm: rel T2 spinc pres}]
Let $X \to E \to T^2$ be the fiber bundle given by \cref{lem: T2 family constrcution}.
By \cref{lem: Fiberwise spinc}, the structure group of $E$ reduces to $\Aut((X,\fraks),\del)$.

Now we calculate $w_2(H^+(E))$.
For $i=1,2$, let $p_i : T^2 \to S^1$ denote the $i$-th projection.
Let $\gamma \to S^1$ denote the non-trivial real line bundle over $S^1$.
Let $\gamma_i \to T^2$ denote the pull-back $p_i^\ast \gamma$.
As in the proof of \cref{thm: rel RP2 spinc pres}, we have that 
\[
w_2(H^+(E))
= w_2(\gamma_1 \oplus \gamma_2)
= w_1(\gamma_1)w_1(\gamma_2) \neq 0.
\]
Hence \lcnamecref{thm: rel T2 spinc pres} follows from \cref{thm: KT}.
\end{proof}

Again we have a version for closed 4-manifolds:

\begin{thm}
\label{thm: T2 spinc pres closed}
Let $X$ be a closed, oriented, smooth 4-manifold with $b^+(X)=2$ and $b_1(X)=0$.
Let $\fraks$ be a spin$^c$ structure on $X$.
Suppose that there are diffeomorphisms $f_1, f_2 \in \Diff^+(X)$ such that the commutator $[f_1, f_2]$ is isotopic to the identity, and $f_1^\ast, f_2^\ast$ act on $H^+(X)$ as $\diag(-1,0), \diag(0,-1)$ for some choice of basis of $H^+(X)$, and $f_i^\ast \fraks \cong \fraks$.
Then  we have
\[
c_1(\fraks)^2-\sigma(X) \leq 0.
\]
\end{thm}

\begin{proof}
We can repeat the proof of  \cref{thm: rel T2 spinc pres} using \cref{thm: Baraglia} in place of \cref{thm: KT}.
\end{proof}

There is also a refinement of \cref{thm: rel T2 spinc pres} for the spin case.
We shall not use this refinement later, and give a sketch of the proof here:

\begin{thm}
\label{thm: spin over torus}
Let $X$ be a compact, oriented, smooth 4-manifold with boundary and with $2 \leq b^+(X) \leq 4$ and $b_1(X)=0$.
Set $Y = \del X$ and suppose that $Y$ is a rational homology 3-sphere.
Let $\fraks$ be a spin structure on $X$.
Suppose that there are diffeomorphisms $f_1, f_2 \in \Diff(X,\del)$ such that $[f_1, f_2]$ is isotopic to the identity through $\Diff(X,\del)$ and $f_i^\ast \fraks \cong \fraks$.
Then:
\begin{enumerate}
\item [(i)] If $b^+(X)=2$ and $f_1^\ast, f_2^\ast$ act on $H^+(X)$ as $\diag(-1,1), \diag(1,-1)$ for some choice of basis of $H^+(X)$, then we have
\begin{align*}
-\frac{\sigma(X)}{8} \leq \gamma(Y,\fraks|_Y).
\end{align*}
\item [(ii)] If $b^+(X)=3$ and $f_1^\ast, f_2^\ast$ act on $H^+(X)$ as $\diag(-1,-1,1), \diag(1,-1,-1)$ for some choice of basis of $H^+(X)$, then we have
\begin{align*}
-\frac{\sigma(X)}{8} \leq \beta(Y,\fraks|_Y).
\end{align*}
\item [(iii)] If $b^+(X)=4$, suppose that the action $(f_1^\ast, f_2^\ast)$ on $H^+(X)$ is one of the following for some choice of basis of $H^+(X)$: 
\begin{align*}
&\left(\diag(-1,-1,-1,-1), \diag(1,-1,-1,-1)\right), \left(\diag(-1,-1,-1,-1), \diag(1,1,1,-1)\right),\\ 
&\left(\diag(-1,1,-1,-1), \diag(1,-1,-1,-1)\right),
\left(\diag(-1,-1,1,-1), \diag(1,1,-1,-1)\right),\\
&\left(\diag(-1,-1,-1,1), \diag(1,1,1,-1)\right).
\end{align*}

Then we have
\begin{align*}
-\frac{\sigma(X)}{8} \leq \alpha(Y,\fraks|_Y).
\end{align*}
\end{enumerate}    
\end{thm}

\begin{proof}
The proof is similar to that of \cref{thm: rel RP2 spin pres}.
Let $\tilde{f}_1, \tilde{f_2} \in \Aut(X,\fraks)$ be lifts of $f_i$ along the covering map $\pi$ in \eqref{eq: covering auto} with $\tilde{f}_i|_{(\del X, \fraks|_{\del X})} = \id_{(\del X, \fraks|_{\del X})}$.
As in the proof of \cref{lem: RP2 family constrcution spin}, one can deduce from $f_i|_{\del X} = \id_{\del X}$ that $[\tilde{f}_1, \tilde{f}_2]=\id_{(X,\fraks)}$.
Thus the pair $(\tilde{f}_1, \tilde{f}_2)$ gives rise to a family $E \to T^2$ of spin 4-manifolds with fiber $(X,\fraks)$ over $T^2$ by replacing $(f_1,f_2)$ with $(\tilde{f}_1, \tilde{f}_2)$ in the proof of \cref{lem: T2 family constrcution}.
It is straightforward to check $w_2(H^+(E))\neq0$ for all the cases in the statement of the \lcnamecref{thm: spin over torus}.
Thus the \lcnamecref{thm: spin over torus} follows from \cref{thm: KT spin}.
\end{proof}

\section{Proofs of the results on Dehn twists}
\label{section Proof of reluts on Dehn twists}

\subsection{Preliminary: Dehn twists and Milnor fibers}
\label{Preliminary: definition of Dehn twists and Milnor fibers}

We recall the Dehn twist emerging from the circle action on the Brieskorn spheres.
We will use notation on Dehn twists introduced in \cref{Intro: Exotic Dehn twists}.
Recall the standard expression of $\Sigma(p,q,r)$:
\[
\Sigma(p,q,r) = \Set{(z_1,z_2,z_3) \in \C^3 | z_1^p+z_2^q+z_3^r=0} \cap S(\C^3).
\]
Unless $(p,q,r) = (2,3,5)$, we have that
\begin{align}
\label{MCG Brieskorn}
\pi_0(\Diff^+(\Sigma(p,q,r))) \cong \Z/2,
\end{align}
generated by the involution $\tau : \Sigma(p,q,r) \to \Sigma(p,q,r)$ defined by $\tau(z_1,z_2,z_3) = (\bar{z}_1, \bar{z}_2, \bar{z}_3)$ (see, for example, \cite[Theorem 7.6]{DHM20}).
For $p=2$,
we denote by 
$\sigma : \Sigma(2,q,r) \to \Sigma(2,q,r)$
the involution defined by $\sigma(z_1,z_2,z_3) = (-z_1,z_2,z_3)$.
Then $\sigma$ is the covering involution associated with the construction of $\Sigma(2,q,r)$ as the double branched covering of $S^3$ along the positive torus knot $T_{q,r}$.
Note that any orientation-preserving involution of $\Sigma(2,q,r)$ is conjugate to $\sigma$ or $\tau$ (see, for example \cite[first Theorem in Section~5]{AnvariHambleton21}).

Set $Y=\Sigma(2,q,r)$.
We can define a smooth isotopy $T : Y \times [0,1] \to Y$ 
from $\id_Y$ to $\sigma$ by
\[
T((z_1,z_2,z_3),s) = e^{is\pi} \cdot (z_1,z_2,z_3),
\]
where $u \cdot (z_1,z_2,z_3)$ for $u \in S^1$ is defined by $(u^{qr}z_1, u^{2r}z_2, u^{2q}z_3)$.
Define 
\begin{align}
\label{eq: TY}
T_Y : Y \times [0,1] \to Y \times [0,1]
\end{align}
by $T_Y(y,s)=(T(y,s),s)$.
Then $T_Y|_{Y \times \{0\}} = \id_Y$, $T_Y|_{Y \times \{1\}} = \sigma$, 
and $(T_Y)^2 \in \Diff(Y \times [0,1],\del)$ is the Dehn twist $t_{Y} : Y \times [0,1] \to Y \times [0,1]$.
If $X$ is a 4-manifold bounded by $Y$, we may define a diffeomorphism 
\begin{align}
\label{eq: T_YX}
T_X : X \to X
\end{align}
with ${T_X}|_{\del X} = \sigma$ by fixing an identification between a collar neighborhood of $X$ and $Y \times [0,1]$, and by extending $T_Y$.
We will also use a diffeomorphism 
\begin{align}
\label{eq: -T}
-T_Y : Y \times [0,1] \to Y \times [0,1]     
\end{align}
defined by $(-T_Y)(y,s) = T_{Y}(y,1-s)$.
Then we have $(-T_Y)|_{Y \times \{0\}} = \sigma$, $(-T_Y)|_{Y \times \{1\}} = \id_Y$.

Now we will recall some basics of Milnor fibers, which is an ingredient of the proof of result on Dehn twists.
Let $M(p,q,r)$ denote the (compactified) Milnor fiber:
\[
M(p,q,r) = \Set{(z_1,z_2,z_3) \in \C^3 | z_1^p + z_2^q + z_3^r = \epsilon} \cap D(\C^3),
\]
where $\epsilon \neq 0$ is small enough.
Recall that $M(p,q,r)$ admits the following description as a branched covering.
(See, for example, \cite[Example~6.3.10]{GS99}.)
Let $S \subset D^4$ be a smoothly embedded surface bounded by the torus link $T_{q,r}$ obtained by pushing the canonical Seifert surface of $T_{q,r}$.
Then $M(p,q,r)$ is the $p$-th branched covering of $D^4$ along $S$.
Under identifying $\del M(p,q,r)$ with $\Sigma(p,q,r)$, this branched covering description is compatible with regarding $\Sigma(p,q,r)$ as the $p$-th branched covering of $S^3$ along $T_{q,r}$.

We will repeatedly use the following result by Orson and Powell in the topological category.
Note that, given a 3-manifold $Y$ (topologically) embedded in a 4-manifold $X$ with a loop $S^1 \to \Homeo(Y)$ based at the identity, we have the Dehn twist $t_X : X \to X$ along $Y$, which lies in $\Homeo(X,\del)$.

\begin{thm}[{Orson and Powell~ \cite[Corollary C]{OrsonPowell2022}}]
\label{thm: OrsonPowell}
Let $X$ be a simply-connected, oriented, compact, topological 4-manifold such that $\del X$ has the rational homology of $S^3$
or of $S^1 \times S^2$.
If a diffeomorphism $f \in \Homeo(X,\del)$ acts trivially on $H_2(X;\Z)$, then $f$ is topologically isotopic to the identity through $\Homeo(X,\del)$.

In particular, given an integral homology 3-sphere $Y$ embedded in $X$ with a loop $S^1 \to \Homeo(Y)$ based at the identity, the Dehn twist $t_X : X \to X$ along $Y$ is topologically isotopic to the identity through $\Homeo(X,\del)$.
\end{thm}

We shall refer to \cref{thm: OrsonPowell} as the {\it Orson--Powell theorem}.

\subsection{Proofs of \cref{main thm1: Dehn twist on definite,contractible}}
Now we will prove a few building blocks of the proof of the result on Dehn twists along $\Sigma(2,3,6n+7)$,  \cref{main thm1: Dehn twist on definite}.
More strongly, we will prove the following:

\begin{thm}
\label{main thm1 generelized: Dehn twist on definite}
Let $n \geq 0$ and let $X$ be a compact oriented positive-definite smooth 4-manifold with $b_1(X)=0$ bounded by $Y=\Sigma(2,3,6n+7)$.
Then the boundary Dehn twist $t_X : X \to X$ is not smoothly isotopic to the identity through $\Diff (X, \partial)$.
\end{thm}

First, we summarize a property of the Milnor fiber $M(2,3,7)$ needed for the proof of \cref{main thm1 generelized: Dehn twist on definite}.
Let $-E_8$ denote the negative-definite, unimodular lattice of rank 8 and $H$ denote the hyperbolic form of size 2, the intersection form of $S^2 \times S^2$.

\begin{lem}
\label{lem: equiv bound}
Let $W_0 = M(2,3,7)$.
Then $W_0$ is a compact oriented simply-connected smooth 4-manifold with the intersection form $-E_8 \oplus 2H$, and there is a smooth involution $\tilde{\sigma}_0 : W_0 \to W_0$ that satisfies the following properties:
\begin{itemize}
\item $\tilde{\sigma}_0$ restricts to $\sigma : \Sigma(2,3,7) \to \Sigma(2,3,7)$.
\item $\tilde{\sigma}_0$ acts on $H^+(W_0)$ via multiplication of $-1$.
\end{itemize}
\end{lem}

\begin{proof}
The intersection form of $M(2,3,7)$ is well-known (see such as \cite[Theorem 8.3.2]{GS99}).
Let $\tilde{\sigma}_0 : W_0 \to W_0$ be the covering involution associated with the above branched covering description of $M(2,3,7)$.
Then $\tilde{\sigma}_0$ restricts to $\sigma$ as described in \cref{Preliminary: definition of Dehn twists and Milnor fibers}.
Since the quotient $W_0/\tilde{\sigma}_0$ is $D^4$, which has trivial $b^+$, we have that $\tilde{\sigma}_0$ acts on $H^+(W_0)$ via multiplication by $-1$.
\end{proof}

The next \lcnamecref{lem: equiv cob} is key to producing infinitely many examples of exotic Dehn twists. We would like to thank the anonymous referee for suggesting a simpler version of the proof.
\begin{lem}\label{lem: equiv cob}
Given positive integers $q > p > 0$ coprime to 6, with $6 | (q - p)$, there is a simply-connected smooth cobordism $W_1$ from $\Sigma(2, 3, p)$ to $\Sigma(2, 3, q)$ that satisfies the following
conditions:
\begin{itemize}
    \item The intersection form of $W_1$ is $(-1)^{\oplus \frac{q-p}{6}}$.
    \item There is an involution $\tilde{\sigma}_1: W_1 \rightarrow W_1$  that restricts to the involution $\sigma$ on both boundary components.
    \item $\tilde{\sigma}$ acts as the identity on $H_2(W_1)$.
\end{itemize}

\end{lem}

\begin{proof}
Without loss of generality, we may assume
$ q - p = 6$ (otherwise, we can just stack the cobordisms together). Recall that the standard circle action makes
\begin{gather*}
    \rho : S^1 \times \Sigma(2, 3, p) \rightarrow \Sigma(2, 3, p) \\
    (\theta, z_1, z_2, z_3) \mapsto (e^{3p \theta i}z_1, e^{2p \theta i}z_2, e^{6 \theta i}z_3)
\end{gather*}
$\Sigma(2,3,p)$ it into a Seifert fibered space with 3 singular fibers of multiplicity 2, 3, $p$ respectively, denoted by $F_1$, $F_2$ and $F_3$. Then $\Sigma(2, 3, p+6)$ is obtained by doing a $(-1)$-surgery along $F_3$. Let $W_1$ be obtained by attaching a 2-handle to $\Sigma(2, 3, p) \times [0, 1]$ along $F_3$. Then $W_1$ has intersection form $(-1)$.

Now we define $\tilde{\sigma}_1$. Firstly, we note that $\sigma : \Sigma(2, 3, p) \rightarrow \Sigma(2, 3, p)$ is just $\rho(\pi,\_)$. In
particular, $\sigma$ fix $F_1$ pointwise and acts as a 2-periodic (i.e. fixing the orientation) $\pi$-rotation on all other fibers. So in a tubular neighbourhood of $F_3$, $\sigma$ can be written as 
\begin{gather*}
\sigma: \mathbb{D}^2 \times S^1 \rightarrow \mathbb{D}^2 \times S^1 \\
 (x,y) \mapsto (x,-y).
\end{gather*}
This action extends over the 2-handle attachment associated to $(-1)$-surgery and gives an involution $\tilde{\sigma}_1$ on $W_1$, see for example \cite[Subsection 5.2]{DHM20}. Moreover, the restriction of $\tilde{\sigma}_1$ on $\Sigma(2,3,p+6)$ also acts as a 2-periodic symmetry, given by $\pi$-rotation on all fibers except $F_1$. Hence $\tilde{\sigma}_1|_{\Sigma(2,3,p+6)}$ also equals $\sigma$ on $\Sigma(2,3,p+6)$. Lastly, since the involution $\tilde{\sigma}_1$ is periodic with respect the fiber $F_3$, it was shown in \cite[Subsection 5.2]{DHM20} that $\tilde{\sigma}_1$ acts as identity on the intersection form of $W_1$.
\end{proof}

Now we are ready to prove a generalization of the statement in \cref{main thm1: Dehn twist on definite} for the smooth category, \cref{main thm1 generelized: Dehn twist on definite}:

\begin{proof}[Proof of \cref{main thm1 generelized: Dehn twist on definite}]
We use equivariant cobordisms $(W_0, \tilde{\sigma}_0)$ and $(W_1, \tilde{\sigma}_1)$ given in \cref{lem: equiv bound,lem: equiv cob} (when $n=0$, we take $W_1$ to be $\Sigma(2,3,7) \times [0,1]$).
Set $Y_0 = \Sigma(2,3,7)$ and $Y_1 = \Sigma(2,3,6n+7)$.
By the conclusion of \cref{lem: equiv cob}, there is a self-diffeomorphism $\psi : Y_1 \to Y_1$ so that $\tilde{\sigma}_1|_{Y_1} = \psi^{-1} \circ \sigma \circ \psi$.
Define a smooth closed 4-manifold $M$ with $b^+(M)=2$ and $b_1(M)=0$ and a diffeomorphism $f : M \to M$ by 
\begin{align*}
M &= W_0 \cup_{\id_{Y_0}} W_1 \cup_{\psi} (-X),\\ 
f &= \tilde{\sigma}_0 \cup_{\id_{Y_0}} \tilde{\sigma}_1 \cup_{\psi} T_X,
\end{align*}
where $T_X$ is the diffeomorphism defined in \eqref{eq: T_YX} and $\del W_1|_{Y_1}$ and $\del(-X)$ is identified by $\psi$ so that $f$ is a well-defined diffeomorphism of $M$.
Note that $H^+(M)$ is isomorphic to $H^+(W_0)$, and hence $f$ acts on $H^+(M)$ via multiplication by $-1$.

Suppose that the boundary Dehn twist $t_X$ on $X$ is isotopic to the identity through $\Diff(X,\del)$.
Then $f^2$ is smoothly isotopic to the identity.
This is because $f^2$ is supported only in a collar neighborhood of $\del(-X)$, and it is given by (the extension by the identity of) the Dehn twist $t_{Y_1}$.
We will obtain a contradiction from this. 

We will define a spin$^c$ structure $\fraks$ on $M$ as follows.
First, note that the intersection form of $X$ is diagonalizable, due to the Elkies theorem~\cite{elkies} combined with that the Fr{\o}yshov invariant of $Y_1$ is zero.
Let $\fraks_0$ be the unique spin structure on $W_0$.
Let $\fraks_1$ be a spin$^c$ structure on $W_1$ that restricts to a generator of each $(-1)$-summand of the intersection form of $W_1$.
Let $\fraks_X$ be a spin$^c$ structure on $X$ that restricts to a generator of each $(-1)$-summand of the intersection form of $X$, under a choice of the basis of $H^2(X;\Z)/{\mathrm{Tor}}$.
We define $\fraks = \fraks_0 \cup \fraks_1 \cup \fraks_X$.

We claim that $f$ preserves $\fraks$.
First, since the spin$^c$ on $Y_i$ for each $i=0,1$ is unique, there is no ambiguity with gluing spin$^c$ structures along $Y_i$.
Also, a spin$^c$ structure on $W_i$ is determined by its first Chern class since $H^2(W_i;\Z)$ has no 2-torsion.
Hence, by construction, $f|_{W_i}$ preserves $\fraks_i$ for each $i=0,1$.
Since $f|_{-X}$ is identity expect for a collar neighborhood of $-X$, it follows that $f|_{-X}$ preserves $\fraks_X$ from that the collar neighborhood of $-X$ has unique spin$^c$ structure.
Thus we have that $f$ preserves $\fraks$.

However, since we have $c_1(\fraks)^2 -\sigma(M) >0$ by a direct computation, this contradicts \cref{thm: RP2 spinc pres}.
This completes the proof.    
\end{proof}

Now we can prove \cref{main thm1: Dehn twist on definite} immediately:

\begin{proof}[Proof of \cref{main thm1: Dehn twist on definite}]
\cref{main thm1: Dehn twist on definite} follows from \cref{main thm1 generelized: Dehn twist on definite} and the Orson--Powell theorem (\cref{thm: OrsonPowell}).
\end{proof}

\begin{rem}
\label{rem: stabilization by CP2}
We note a consequence of \cref{main thm1: Dehn twist on definite} on the stabilization by $\CP^2$.
\cref{main thm1: Dehn twist on definite} implies that the Dehn twist $t_X$ remains relatively exotic after taking (inner) connected sum of $X$ with arbitrarily many finite copies of $\CP^2$. 
In particular, it shows that the Dehn twist along a general 3-manifold needs not to be (smoothly and) relatively isotopic to the identity after connected sums of $\CP^2$.
In contrast, the Dehn twist along $S^3$ turns relatively isotopic to the identity after the connected sum of a single copy of $\CP^2$ due to Giansiracusa~\cite[Corollary~2.5]{Gian08}.
\end{rem}



\begin{proof}
The proof is similar to that of \cref{main thm1: Dehn twist on definite}.
\end{proof}

This is now a good place to prove a result on Dehn twists on closed 4-manifolds, \cref{thm: closed Dehn}.

\begin{proof}[Proof of \cref{thm: closed Dehn}]
Let $X$ be a simply-connected smooth 4-manifold bounded by $Y=\Sigma(2,3,6n+7)$ with intersection form $(+1)$.
We may get such $X$ by regarding $Y$ as a $(+1)$-surgery of a twist knot and taking the surgery trace.
Let $(M,f)$ be as in the proof of \cref{main thm1 generelized: Dehn twist on definite}.
Then it is easy to see that the intersection form of $M$ is $(1)^{\oplus 2} \oplus (-1)^{\oplus (n+11)}$.
As an easy consequence of Van-Kampen's theorem, we can see that $M$ is simply-connected.
Thus it follows from Freedman theory that $M$ is homeomorphic to $2\CP^2\#(n+11)(-\CP^2)$.

By constrcution, $Y_1=\Sigma(2,3,6n+7)$ is embedded in $M$.
Set $M'=M\#k(-\CP^2)$ for $k \geq 0$.
The proof of \cref{main thm1 generelized: Dehn twist on definite} applied to $M'$ instead of $M$ implies that $f^2$, which coincides with the Dehn twist along $Y$, is not smoothly isotopic to the identity on $M'$.
\end{proof}

\subsection{Brieskorn spheres $\Sigma(2,3,6n+11)$}
\label{subsec: 12n-1}

Now we prove a result on the family $\Sigma(2,3,6n+11)$ (\cref{thm: More Dehn twists}).
To prove this, we will need a few lemmas:

\begin{lem}
\label{lem: equiv bound2}
Let $W_0 = M(2,3,11)$.
Then $W_0$ is a compact oriented simply-connected smooth 4-manifold with the intersection form $2(-E_8) \oplus 2H$, and there is a smooth involution $\tilde{\sigma}_0 : W_0 \to W_0$ that satisfies the following properties:
\begin{itemize}
\item $\tilde{\sigma}_0$ restricts to $\sigma : \Sigma(2,3,11) \to \Sigma(2,3,11)$.
\item $\tilde{\sigma}_0$ acts on $H^+(W_0)$ via multiplication of $-1$.
\end{itemize}
\end{lem}

\begin{proof}
As in the proof of \cref{lem: equiv bound}, the intersection form of $M(2,3,11)$ is well-known. For example see \cite[Corollary 7.3.23]{GS99}. 
Let $\tilde{\sigma}_0 : W_0 \to W_0$ be the covering involution associated with the branched covering description of $M(2,3,11)$.
Then $\tilde{\sigma}_0$ satisfies the desired properties.
\end{proof}

\begin{proof}[Proof of \cref{thm: More Dehn twists}]
The proof is again similar to the proof of \cref{main thm1: Dehn twist on definite}.
We use equivariant cobordisms $(W_1, \tilde{\sigma}_1)$ and $(W_0, \tilde{\sigma}_0)$ given in \cref{lem: equiv bound2,lem: equiv cob}.
Set $Y_0 = \Sigma(2,3,11)$ and $Y_1 = \Sigma(2,3,6n+11)$.
Then there is a self-diffeomorphism $\psi : Y_1 \to Y_1$ so that $\tilde{\sigma}_1|_{Y_1} = \psi^{-1} \circ \sigma \circ \psi$.

Let $-T_{Y_1} : Y_1 \times [0,1] \to Y_1 \times [0,1]$ be the diffeomorphism defined by \eqref{eq: -T} for $Y=Y_1$.
Define a smooth 4-manifold $X$ with $b^+(X)=2$ and a diffeomorphism $f : X \to X$ by 
\begin{align*}
X &= W_0 \cup_{\id_{Y_0}} W_1 \cup_{\psi} (Y_1 \times [0,1]),\\ 
f &= \tilde{\sigma}_0 \cup_{\id_{Y_0}} \tilde{\sigma}_1 \cup_{\psi} (-T_{Y_1}),
\end{align*}
where $\del W_1|_{Y_1}$ and $Y_1 \times \{0\}$ is identified by $\psi$ so that $f$ is a well-defined diffeomorphism of $X$.
Note that $H^+(X)$ is isomorphic to $H^+(W_0)$, and hence $f$ acts on $H^+(X)$ via multiplication by $-1$. 
Since $W_1$ consists only of 2-handles, the inclusion $Y_0 \to W_1$ induces a surjection $\pi_1(Y_0) \to \pi_1(W_1)$. Since we also know that $W_0$ is simply connected, Van-Kampen's theorem implies $X$ is also simply connected. 

Suppose that the boundary Dehn twist $t_X$ on $X$ is isotopic to the identity through $\Diff(X,\del)$.
Then $f^2$ is smoothly isotopic to the identity.
We will obtain a contradiction from this. 

Let $\fraks$ be the spin$^c$ structure on $X$ determined by 
\[
c_1(\fraks) = (0,-1,\ldots, -1),
\]
where $0$ is the zero of $H^2(W_0)$ and $(-1,\ldots, -1)$ is an element in $H^2(W_1;\Z)$ under a choice of basis.
Note that $f$ preserves $\fraks$ by construction.
However, since $c_1(\fraks)^2 -\sigma(X) >0$, this contradicts \cref{thm: rel RP2 spinc pres}.
This combined with the Orson--Powell theorem (\cref{thm: OrsonPowell}) completes the proof.
\end{proof}

\subsection{Other examples}
\label{subsection 279}
For the examples $\Sigma(2,3,7)$ and $\Sigma(2,3,11)$, the Milnor fibers of them are effectively used to construct interesting families of 4-manifolds parametrized by $\RP^2$. In this subsection, we give a way to use surface cobordism theory to find families of 4-manifolds. 

Let us first review the twisting operation for knots. 
For given knot $K$ in $S^3$ and an embedded 2-dimensional disk $D=D^2$ in $S^3$ which intersects with $K$ transversely at the interior of $D$, we do the $(-1/n)$-surgery along $\partial D$ in $S^3$. 
Since $\partial D$ is the unknot in $S^3$, $S^3_{-1/n}(\partial D)$ is diffeomorphic to $S^3$ (by a particular choice of a diffeomorphism) and thus we have another knot $K(D, n)$ as the image of the knot $K$ under the diffeomorphism. The operation $K \mapsto K(D, n)$ is called the {\it $(n,\operatorname{lk}(\partial D, K))$-twisting operation}, where $\operatorname{lk}(\partial D, K)$ denotes the linking number between $\partial D$ and $K$. In terms of knot theory, $K(D, n)$ is obtained as the $n$-full twist of strands along $\partial D$. 
We will use the following result proven in \cite{KY00}: 
\begin{lem}\cite[Lemma 4.4]{KY00} Suppose $n = \pm 1$.
    If there is a $(n,s)$-twisting operation from $K$ to $K'$, then there is a smooth annulus cobordism $S$ from $K$ to $K'$ inside $(-n) \CP^2 \setminus \operatorname{Int}D^4 \cup \operatorname{Int}D^4$ whose homology class is 
    \[
    [S] = s [\CP^1] \in H_2((-n) \CP^2 \setminus \operatorname{Int}D^4 \cup \operatorname{Int}D^4 ; \Z ).  
    \]
\end{lem}

It is proven in \cite[Theorem 1.1(1)]{LP22} that there is a $(-1,mn)$-twisting move from $T_{mn + m + 1, mn + 1}$ to $T_{mn + n + 1, m + 1}$ for every positive integers $m,n >0$.  The case $m=2, n=3$ gives $(-1,6)$-twisting operation from $T_{9, 7}$ to $T_{10, 3}$. 
It gives a smooth annulus cobordism from $T_{9, 7}$ to $T_{10, 3}$ in $\CP^2 \setminus (\operatorname{Int}D^4 \cup \operatorname{Int}D^4 )$ whose homology class is $6 [\CP^1]$.  By reversing the orientation, we get a smooth annulus cobordism $S$ from  $T_{10, 3}$ to $T_{9, 7}$ in $-\CP^2 \setminus \operatorname{Int}D^4 \cup \operatorname{Int}D^4 $ whose homology class is $6 [\CP^1]$.
Since $T_{10,3}$ bounds a smoothly embedded orientable surface $S'$ of genus $9$ in $D^4$, by capping off $S$ by $S'$, we obtain a smoothly and properly embedded surface $S^\#$ in   $-\CP^2 \setminus \operatorname{Int}D^4 $ bounded by $T_{9,7}$. Then we denote by $\Sigma(S^\#)$ the double branched covering space along $S^\#$. One can check $b^+(\Sigma(S^\#))=2$, $\sigma(\Sigma(S^\#)) =-16$ and $\Sigma(S^\#)$ is simply-connected and spin.  It is proven \cite[Table 1]{Tw13} that the Heegaard Floer $d$-invariant for $\partial \Sigma(S^\#) = \Sigma(2,7,9)$ satisfies $d(\Sigma(2,7,9))=2$. Using equivalence between the $d$-invariant and the Seiberg--Witten Fr\o yshov invariant \cite[Remark 1.1]{LRS18}, we get $\delta(\Sigma(2,7,9))=1$. 
Thus, by using $\Sigma(S^\#)$ instead of the Milnor fiber, one can repeat the argument given in \cref{main thm1 generelized: Dehn twist on definite} and show the Dehn twist along $\Sigma(S^\#)$ is relatively exotic. 
In fact, this relatively exotic Dehn twist survives after one stabilization, see \cref{one stabilization 279}.

\subsection{Proof of the results on one stabilization}

Let us now prove the result on exotic diffeomorphism surviving stabilization, \cref{thm: intro stabilization in K3}.

\begin{proof}[Proof of \cref{thm: intro stabilization in K3}]
From the Orson--Powell theorem (\cref{thm: OrsonPowell}), it is enough to prove that the Dehn twists along $\Sigma(2,3,7)$ and $\Sigma(2,3,11)$ are not smoothly isotopic to the identity through $\Diff(\mathring{K3}\#S^2\times S^2)$.

Firstly, let us consider $M(2,3,7)$.
Set $M=M(2,3,7)$, $N=K3 \setminus \mathop{\mathrm{Int}}M$, $Y=\del M$.
Let us consider the involution $\tilde{\sigma} = \tilde{\sigma}_0 : M \to M$ given in \cref{lem: equiv bound}.
First we will consider the case that $S^2 \times S^2$ is attached inside $\mathop{\mathrm{Int}}N$ when we form $Z :=\mathring{K3}\# S^2\times S^2$.
Let $-T_Y : Y \times [0,1] \to Y \times [0,1]$ be the diffeomorphism defined by \eqref{eq: -T}.
Fix a diffeomorphism $M \cong M \cup (Y \times [0,1])$, and let $f : Z \to Z$ be the diffeomorphism defined by extending $\tilde{\sigma} \cup (-T_Y)$ by the identity of $Z \setminus M$.
Then $f^2$ is the Dehn twist along $\Sigma(2,3,7)$, and $f$ acts on $H^+(Z)$ as $\diag(1,1,-1,-1)$ for an appropriate basis of $H^+(Z)$, thus $b^+(Z)-\dim H^+(Z)^f = 2$.

Let $\fraks$ be the unique spin structure on $Z$, which is of course preserved by $f$.
Thus, if the Dehn twist along $\Sigma(2,3,7)$ is isotopic to the identity through $\Diff(Z,\del)$,
(iii) of \cref{thm: rel RP2 spin pres} implies that $-\sigma(Z)/8 \leq \alpha(S^3)$, but this is a contradiction since $\alpha(S^3)=0$ \cite{Ma16}.

Next, let us consider the case that $S^2 \times S^2$ is attached inside $M$ when we form $Z =\mathring{K3}\# S^2\times S^2$.
Recall that the involution
$\tilde{\sigma}$ was constructed as a covering involution associated to a double branched covering of $D^4$ along a surface in $D^4$.
In particular, the fixed-point set of $\tilde{\sigma}$ is non-empty and of codimension-2.
Let $\iota : S^2 \times S^2 \to S^2 \times S^2$ be an orientation-preserving involution defined by $\iota(x,y)=(y,x)$ for $(x,y) \in S^2 \times S^2$.
Then $\iota$ acts trivially on $H^+(S^2 \times S^2)$, and the fixed-point set of $\iota$ is non-empty and of codimension-2.
Choosing points in the fixed-point sets of $\tilde{\sigma}$ and $\iota$, we may form equivariant connected sum $\tilde{\sigma}\# \iota : M \# S^2 \times S^2 \to M \# S^2 \times S^2$.
Again, fix a diffeomorphism $M \cong M \cup (Y \times [0,1])$, and let $g : Z \to Z$ be the diffeomorphism defined by extending $\tilde{\sigma} \cup \iota \cup (-T_Y)$ by the identity of $Z \setminus M$.
Then $g^2$ is the Dehn twist along $\Sigma(2,3,7)$, and $g$ acts on $H^+(Z)$ as $\diag(1,1,-1,-1)$, so the above proof for the case that $S^2 \times S^2$ is attached inside $\mathop{\mathrm{Int}}N$ works as well by replacing $f$ with $g$.

The case for $M=M(2,3,11)$ is almost identical: we may just use the involution $\tilde{\sigma} = \tilde{\sigma}_0 : M \to M$ given in \cref{lem: equiv bound2} instead of \cref{lem: equiv bound}.
This completes the proof.
\end{proof}

Kronheimer and Mrowka \cite{KM20} proved that the Dehn twist on $K3\# K3$ along $S^3$ in the neck is not isotopic to the identity, and
Lin~\cite{JL20} proved that this is not isotopic to the identity even on $K3\# K3 \# S^2 \times S^2$.
As a consequence, the boundary Dehn twist of $\mathring{K3} \# S^2 \times S^2$, the one stabilization of the punctured $K3$, is not relatively isotopic to the identity, which combined with the Orson--Powell theorem (\cref{thm: OrsonPowell}) implies that the boundary Dehn twist on $\mathring{K3}$ is relatively exotic even after one stabilization.
We remark that we can give an alternative proof of this part of Lin's result:

\begin{thm}[{Lin \cite{JL20}}]
\label{thm Lin}
The boundary Dehn twist of $\mathring{K3} \# S^2 \times S^2$ is not isotopic to the identity through $\Diff(\mathring{K3} \# S^2 \times S^2, \del)$.
\end{thm}

\begin{proof}
Let $\iota_1 : K3 \to K3$ be a smooth involution that acts on $H^+(K3)$ as $\mathrm{diag}(-1,-1,1)$ with codimention-2 non-empty fixed-point set.
An example of such $\iota_1$ is obtained as the covering involution of $K3$ regarded as a double branched cover of $\CP^2$ (see such as \cite[Corollary 7.3.25]{GS99}).
Take a smooth orientation-preserving involution  $\iota_2 : S^2 \times S^2 \to S^2 \times S^2$ with codimension-2 fixed-point set, such as the product of two copies of reflection of $S^2$ along an equator.
From $\iota_1, \iota_2$, we may form an equivariant connected sum along fixed points to get a smooth involution $\iota = \iota_1\#\iota_2 : K3 \# S^2 \times S^2 \to K3 \# S^2 \times S^2$.

Set $X=\mathring{K3}\# S^2 \times S^2$.
Let $f \in \Diff(X,\del)$ be a diffeomorphism which is obtained by deforming $\iota$ by isotopy to fix a 4-disk in $K3$ pointwise.
Then, through $\Diff(X, \del)$,  $f^2$ is isotopic to either the identity or the boundary Dehn twist of $X$ by \cite[Corollary~2.5]{Gian08}.
However, (iii) of \cref{thm: rel RP2 spin pres} implies that $f^2$ is not isotopic to the identity through $\Diff(X, \del)$.
Hence the boundary Dehn twist is isotopic to $f^2$ through $\Diff(X, \del)$, which is not isotopic to the identity through $\Diff(X, \del)$. 
\end{proof}

\begin{rem}\label{one stabilization 279}
Let $X=\Sigma(S^\#)$ be the double-branched covering space used in \cref{subsection 279}. 
By a similar argument and applying \cref{thm: rel RP2 spin pres} to $X \# S^2\times S^2$, if the Dehn twist $t_X$ along $\Sigma(2,7,9)$ is smoothly isotopic to the identity rel boundary, we have 
\[
2=-\frac{1}{8}\sigma (\Sigma(S^\#) \# S^2\times S^2) \leq \beta (\Sigma(2,7,9)).
\]
On the other hand, \cite[Table in page 248]{Sto20} claims $\beta (\Sigma(2,7,9))=0$ which is a contradiction. Thus, one stabilization is not enough to kill the exotic Dehn twist of $X$ along $\Sigma(2,7,9)$.

\end{rem}

\subsection{Candidates of exotic surfaces from the Dehn twists}\label{exotic sur}
In general, if we have exotic diffeomorphisms, then as the images of surfaces by such diffeomorphisms, one can provide candidates of exotic surfaces.
In certain situations, these indeed give exotic surfaces whose complements are diffeomorphic. See \cite{Ba20, LM21, IKMT22} for examples. 
From exotic Dehn twists, we give many candidates of exotic surfaces from the Dehn twists. We also relate this story to the Smale conjecture. 
Given smoothly embedded surfaces $S, S'$ in a 4-manifold $X$ with boundary such that $\del S, \del S' \subset \del X$, we say that $S$ and $S'$ are {\it relatively exotic} they are topologically isotopic rel boundary but not smoothly isotopic rel boundary.

\begin{prop}
\label{prop exotic surface}
At least one of the following (i), (ii) is true: 
\begin{itemize}
\item[(i)] The homomorphism 
    \[
    \pi_0 (\operatorname{Diff}(D^4, \partial )) \to \pi_0 (\operatorname{Homeo}(D^4, \partial )) 
    \]
is not injective.
\item [(ii)] Let $Y$ be a Seifert fibered homology 3-sphere.
Let $X$ be a geometrically simply-connected 4-manifold $X$ bounded by $Y$, i.e. $-X$ has a handle decomposition consisting only one $4$-handle and $2$-handles attached to $-Y$. 
Let $S$ be the union of all $2$-handle cores of a handle decomposition of $X$. 
Then, if the boundary Dehn twist $t_X$ on $X$ is relatively exotic, the surfaces $S$ and $t_X(S)$ are relatively exotic. 
\end{itemize}
\end{prop}



\begin{proof}
Note that $-X$ can be decomposed into 
\[
[0,1]\times (-Y) \cup \text{(2-handles)} \cup \text{(a 4-handle)}.
\]
Since $t_X$ is topologically isotopic to the identity relative to the boundary, we know that $S$ and $t_X(S)$ are topologically isotopic relative to the boundary. Suppose $S$ and $t_X(S)$ are smoothly isotopic relative to the boundary.
Then, by the isotopy extension lemma, one can construct a relative smooth isotopy from the identity to some diffeomorphism $f$ on $-X$ such that 
$S = f( t_X(S))$. 
Then, by applying \cite[Lemma 3.1]{LM21} to $f \circ t_X$, we can further make an isotopy rel boundary between $f \circ t_X$ and an orientation preserving diffeomorphism $g: X  \to X$ satisfying $g|_{X \setminus \operatorname{Int} D^4}= \id$. Thus, $g|_{D^4}$ gives an element in $\pi_0 (\operatorname{Diff}(D^4, \partial ))$. Suppose $\pi_0 (\operatorname{Diff}(D^4, \partial )) \to \pi_0 (\operatorname{Homeo}(D^4, \partial )) $ is injective. Then, we see $g|_{D^4}$ is smoothly and relatively isotopic to $\id$.
Thus, we conclude that $t_X$ is relatively and smoothly isotopic to $\id$, which contradicts with the first condition. 
\end{proof}
Since we know that the boundary Dehn twist $t_X$ is exotic for any simply-connected positive-definite 4-manifold $X$ bounded by $\Sigma(2,3,6n+7)$, combined with \cref{prop exotic surface}, we have the following. 

\begin{cor}
At least one of the following (i), (ii) is true:
\begin{itemize}
\item[(i)] The homomorphism 
\[
    \pi_0 (\operatorname{Diff}(D^4, \partial )) \to \pi_0 (\operatorname{Homeo}(D^4, \partial )) 
\]
is not injective.
\item [(ii)] For any geometrically simply-connected positive-definite 4-manifold $X$ bounded by $\Sigma(2,3,6n+7)$ for $n \geq 0$, the union of the 2-handle cores $S$ and $t_X(S)$ are relatively exotic
\end{itemize}
\end{cor}


\section{Loops of diffeomorphisms}
\label{section Loops of diffeomorphisms}

Now we rephrase our results on Dehn twists in a language of fundamental groups of automorphism groups, following Orson and Powell~\cite{OrsonPowell2022}.

Let $X$ be a smooth 4-manifold with an orientable compact boundary $Y$.
Then we have a fiber sequence
\[
\Diff(X,\del) \stackrel{\iota}{\hookrightarrow} \Diff(X) \stackrel{r}{\to} \Diff(Y).
\]
Here $\iota$ is the inclusion and $r$ is the restriction.
Similarly, we have a fibration 
\[
\Homeo(X,\del) \stackrel{\iota}{\hookrightarrow} \Homeo(X) \stackrel{r}{\to} \Homeo(Y).
\]
See \cite[Proposition A.2]{OrsonPowell2022} for the proof that these are the fibrations (the proof is given in the topological category, but it is easy to modify the proof to get the smooth version.)
Together with the inclusion maps $i : \Diff(X) \hookrightarrow
\Homeo(X)$ et cetera from diffeomorphism groups to homeomorphism groups, we obtain from the above fibrations a commutative diagram
\begin{equation}
\label{eq: big commutative diagram}
\begin{tikzcd}
  \cdots \arrow[r] & \pi_1(\Diff(X)) \arrow[r, "r_\ast"] \arrow[d, "i_{\ast}"] & \pi_1(\Diff(Y)) \arrow[r, "\del_{C^\infty}"] \arrow[d, "\cong"] & \pi_0(\Diff(X,\del)) \arrow[r] \arrow[d] & \cdots \\
  \cdots \arrow[r] & \pi_{1}(\Homeo(X)) \arrow[r, "r_\ast"] & \pi_{1}(\Homeo(Y)) \arrow[r, "\del_{C^0}"] & \pi_{0}(\Homeo(X,\del)) \arrow[r] & \cdots.
\end{tikzcd}    
\end{equation}

Here $\del_{C^\infty}$ and $\del_{C^0}$ are connecting homomorphisms.
Note that the induced map $\pi_1(\Diff(Y)) \to \pi_1(\Homeo(Y))$ is isomorphic since $\Diff(Y) \hookrightarrow \Homeo(Y)$ is a weak homotopy equivalence (see, such as, \cite{Hat80} together with \cite{Ce68,Ha83}).

\begin{lem}
\label{lem: pi 1 comparison and extension}
Let $X$ be a smooth 4-manifold with an orientable compact boundary $Y$.
If there exists an element $\gamma \in \pi_1(\Diff(Y)) (\cong \pi_1(\Homeo(Y)))$ such that
\begin{align*}
&\gamma \in
\im(r_\ast : \pi_1(\Homeo(X)) \to \pi_1(\Homeo(Y))),\\
&\gamma \notin \im(r_\ast : \pi_1(\Diff(X)) \to \pi_1(\Diff(Y))),
\end{align*}
then $i_\ast : \pi_1(\Diff(X)) \to \pi_1(\Homeo(X))$ is not surjective.
\end{lem}

\begin{proof}
Let $\tilde{\gamma} \in \pi_1(\Homeo(X))$ be a loop that is sent to $\gamma$ under $r_\ast : \pi_1(\Homeo(X)) \to \pi_1(\Homeo(Y))$.
One sees that $\tilde{\gamma}$ does not lies in the image of $i_\ast : \pi_1(\Diff(X)) \to \pi_1(\Homeo(X))$, using that $\gamma$ does not lies in the image of $r_\ast : \pi_1(\Diff(X)) \to \pi_1(\Diff(Y))$ and the commutativity of the diagram \eqref{eq: big commutative diagram}.
\end{proof}

Now let us consider Seifert fibered 3-manifolds.
As in \cref{Non-extendable loops of diffeomorphisms},
for a Seifert fibered 3-manifold $Y(\neq S^3)$,
let $\gamma_S \in \pi_1(\Diff(Y))$ denote the loop that emerges from the standard circle action on $Y$.
We call $\gamma_S$ the {\it Seifert loop} for $Y$, which is non-trivial in $\pi_1(\Diff(Y)) \cong \pi_1(\Homeo(Y))$ \cite[Proposition~8.8]{OrsonPowell2022}.

\begin{lem}
\label{lem: exact}
Let $X$ be a smooth 4-manifold with a Seifert fibered 3-manifold boundary $Y$.
Then the Seifert loop $\gamma_S \in \pi_1(\Diff(Y))$ lies in the image of $r_\ast : \pi_1(\Diff(X)) \to \pi_1(\Diff(Y))$ if and only if $[t_X] \in \pi_0(\Diff(X,\del))$ is trivial.
An analogous statement holds also in the topological category.
\end{lem}

\begin{proof}
It is straightforward to see that $\del_{C^\infty}(\gamma_S)$ coincides with the relative mapping class $[t_X] \in \pi_0(\Diff(X,\del))$ of the boundary Dehn twist $t_X$.
Similarly, $\del_{C^0}(\gamma_S)$ coincides with the class $[t_X] \in \pi_0(\Homeo(X,Y))$, considered in the topological category.
Thus the assertion follows from the exactness of the rows in the diagram \eqref{eq: big commutative diagram}.
\end{proof}

\begin{rem}
\label{rem: equiv plumbing}
An equivariant plumbing technique tells that any Seifert fibered homology 3-sphere $Y$ bounds a negative-definite complex surface $X$ (obtained as a resolution of a singularity) to which the Seifert circle action extends smoothly (\cite{Orlik-Philip}, see also \cite[Subsection~7.2]{Saveliev99}).
In particular, the Seifert loop $\gamma_S$ extends smoothly to such $X$.
By \cref{lem: exact}, this implies that the boundary Dehn twist on $X$ is smoothly isotopic to the identity rel boundary.
\end{rem}

A result by Saeki~\cite{Saeki06} immediately implies the following (see also a recent generalization by Orson and Powell
\cite[Theorem~B]{OrsonPowell2022}):

\begin{thm}[{Saeki
\cite{Saeki06}}]
\label{thm Dehn twist stably trivial}
Let $Y$ be a Seifert fibered homology 3-sphere.
Let $X$ be an oriented compact simply-connected smooth 4-manifold bounded by $Y$.
Then there exists $n\geq 0$ such that:
\begin{itemize}
\item [(i)] The boundary Dehn twist $t_{X\#n S^2 \times S^2} : X\#n S^2 \times S^2 \to X\#n S^2 \times S^2$ is smoothly isotopic to the identity through $\Diff(X,\del)$.
\item [(ii)] The Seifert loop $\gamma_S$ lies in the image of
\[
r_\ast : \pi_1(\Diff(X\#n S^2\times S^2)) \to \pi_1(\Diff(Y)).
\]
\end{itemize}
\end{thm}

\begin{proof}
(i) is equivalent to (ii) by \cref{lem: exact}, and we will check (i).
Note that the boundary Dehn twist $t_X$ on $X$ trivially acts on homology.
This implies that the variation map $\Delta_{t_X} : H_2(X;\Z) \cong H_2(X, \del X;\Z) \to H_2(X;\Z)$  defined in \cite[Definition 2.1]{Saeki06} is the identity map. 
Noting this, (i) is a direct consequence of \cite[Theorem 2.2]{Saeki06}.
\end{proof}

Now we are ready to prove:

\begin{thm}
\label{main thm1 generelized: Dehn twist on definite 2}
Let $X$ be a compact oriented positive-definite smooth 4-manifold with $b_1(X)=0$ bounded by one of $\Sigma(2,3,6n+7)$ for $n \geq 0$.
Set $Y=\del X$.
Then the Seifert loop $\gamma_S \in \pi_1(\Diff(Y))$ does not extend to $X$ smoothly.
In particular:
\begin{itemize}
\item[(i)] The restriction map 
$r_\ast : \pi_1(\Diff(X)) \to \pi_1(\Diff(Y))
$
is not surjective.
\item[(ii)] The standard circle action on $Y$ does not extend to $X$ as a smooth circle action.
\end{itemize}
\end{thm}

\begin{proof}
This is a direct consequence of \cref{main thm1 generelized: Dehn twist on definite,lem: exact}.
\end{proof}

Similarly, we shall prove:

\begin{thm}
\label{thm: pi1 2}
For $n \geq 0$, the Brieskorn sphere $Y=\Sigma(2,3,6n+11)$ bounds a simply-connected oriented compact smooth 4-manifold $X$ such that neither the natural map
\[
\pi_1(\Diff(X)) \to \pi_1(\Homeo(X))
\]
nor the restriction map
\[
r_\ast : \pi_1(\Diff(X)) \to \pi_1(\Diff(Y))
\]
are surjective.   
More precisely, a non-trivial element in the cokernel of $r_\ast$ is given by the Seifert loop.
\end{thm}

\begin{proof}[Proof of \cref{thm: pi1 1}, \cref{thm: pi1 smooth vs. top 1}, \cref{cor: pi1 non surj contractible}, \cref{thm pi 1 Milnor necleus}, \cref{thm: pi1 2}]
These facts follow from the theorems that have already been proven.
\begin{itemize}
    \item \cref{thm: pi1 1} follows from  \cref{main thm1 generelized: Dehn twist on definite 2,thm: OrsonPowell,lem: exact}.
    \item \cref{thm: pi1 smooth vs. top 1} follows from \cref{thm: pi1 1,lem: pi 1 comparison and extension}.
    \item \cref{cor: pi1 non surj contractible} follows from \cref{main thm1 generelized: Dehn twist on definite 2,thm: OrsonPowell,lem: exact}.
    \item \cref{thm pi 1 Milnor necleus} follows from \cref{thm: intro stabilization in K3,lem: pi 1 comparison and extension,lem: exact}. 
    \item \cref{thm: pi1 2} follows from \cref{thm: More Dehn twists,lem: pi 1 comparison and extension,lem: exact}. 
\end{itemize}

\end{proof}

Recall that the notion of strong cork \cite[Section 1.2.2]{LRS18}.
A cork $(Y,\tau)$ is called a {\it strong cork} if $\tau$ does not extend to any smooth ($\Z/2$-)homology 4-ball bounded by $Y$ as a diffeomorphism, and
$\tau$ is said to be {\it strongly non-extendable} in this case.
In a similar spirit, we define the notion of {\it strongly non-extendable loop} of diffeomorphisms.
We may treat rational homology 4-balls not only $\Z$- or $\Z/2$-homology 4-balls, so we define this notion with rational coefficient:

\begin{defn}
\label{defn: storngly non-extendable loop}
Let $Y$ be an oriented closed 3-manifold.
Suppose that $Y$ bounds a contractible compact smooth 4-manifold.
We say that a (homotopy class of) loop of  $\gamma \in \pi_1(\Diff(Y))$ is a {\it strongly non-extendable loop} of diffeomorphisms if $\gamma$ does not lie in the image of the restriction map
\[
r_\ast : \pi_1(\Diff(X)) \to \pi_1(\Diff(Y))
\]
for {\it any} smooth rational homology 4-ball $X$ bounded by $Y$.
\end{defn}

With this term, we can phrase a part of \cref{main thm1 generelized: Dehn twist on definite 2} as:

\begin{cor}
\label{strong pi1 cork}
Let $Y$ be $\Sigma(2,3,13)$ or $\Sigma(2,3,25)$.
Then the Seifert loop $\gamma_S \in \pi_1(\Diff(Y))$ is strongly non-extendable.
\end{cor}

\section{Proofs of results on commutators}
\label{section Proof of reluts on commutators}

\begin{proof}[Proof of \cref{thm: relative commutator}]
Let $m=2n$ for $n>0$ or $m=2n+1$ for $n \geq 0$.
Set $Y=S^3_{1/m}(K)$.
Let $W_m(K)$ be the 4-manifold bounded by $Y$ given by the Kirby picture in Figure~\ref{kirby_1}.
The intersection form of $W_m(K)$ is $H$ is $m$ is even, and $\diag(1,-1)$ if $m$ is odd.
\begin{figure}[h!]
\center
\includegraphics[scale=1.0]{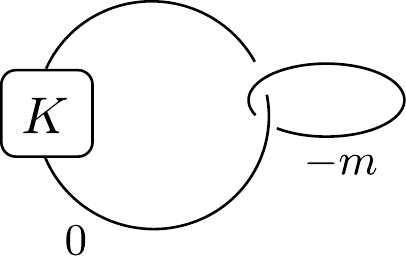}
\caption{The 4-manifold $W_m(K)$.}
\label{kirby_1}
\end{figure}
Define
\[
X = 
\begin{cases}
W_m(K) \# S^2 \times S^2 &\text{  if $m$ is even,} \\
W_m(K) \# S^2 \times S^2 \# (-\CP^2) &\text{  if $m$ is odd.}
\end{cases}
\]
We will prove that $X$ admits a relatively exotic diffeomorphism.

First, when $m$ is even, define $A_i \in \Aut(H_2(X;\Z))$ for $i=1,2$ by 
\begin{align*}
A_1 &= -\id_{H_2(W_m(K))} \oplus \id_{H_2(S^2 \times S^2)},\\
A_2 &= \id_{H_2(W_m(K))} \oplus (-\id_{H_2(S^2 \times S^2)}).
\end{align*} 
When $m$ is odd, fix an isomorphism between $H_2(W_m(K)\#(-\CP^2);\Z)$ and $(-1) \oplus H$. Define $A_i \in \Aut(H_2(X;\Z))$ for $i=1,2$ by 
\begin{align*}
A_1 &= \id_{(-1)} \oplus (-\id_{H}) \oplus \id_{H_2(S^2 \times S^2)},\\
A_2 &= \id_{(-1)} \oplus \id_{H} \oplus (-\id_{H_2(S^2 \times S^2)}).
\end{align*} 

Recall that Wall's theorem on realization for diffeomorphisms can be extended for 4-manifolds with integral homology 3-sphere boundary \cite[Theorem~2, page 136]{Wa64}.
Moreover, the diffeomorphism can be chosen so that it fixes the boundary pointwise (see also  \cite[Corollary~2.2]{RS22}).
Thus we can find $f_i \in \Diff(X,\del)$ such that $(f_i)_\ast=A_i$.

Put $f=[f_1,f_2]$.
We shall prove that $f$ is a relatively exotic diffeomorphism.
First, since $f_\ast=1$, it follows from the Orson--Powell theorem (\cref{thm: OrsonPowell}) that $f$ is topologically isotopic to the identity through $\Homeo(X,\del)$.
Hence it suffices to prove that $f$ is not isotopic to the identity through $\Diff(X,\del)$.

Let $\fraks$ be a spin$^c$ structure on $X$ determined as follows:
If $m$ is even, let $\fraks$ be the unique spin structure.
If $m$ is odd, let $\fraks$ be the spin$^c$ structure with
$c_1(\fraks) = (-1,0,0)$
along the decomposition of $H^2(X)$ to be $(-1) \oplus H \oplus H^2(S^2 \times S^2)$.
Note that $\fraks$ is preserved by both $f_i$ by construction.

Suppose that $f$ is isotopic to the identity through $\Diff(X,\del)$.
Then it follows from \cref{thm: rel T2 spinc pres} that
\[
0=\frac{c_1(\fraks)^2-\sigma(X)}{8} \leq \delta(Y).
\]
However, we have that $\delta(Y)<0$ from the assumption that $V_0(K)>0$, since $2V_0(K)= -d(S_{1/m}(K))=2 \delta(S_{1/m}(K))$. The first equality follows from the work of \cite{rasmussen2003floer, ni2015cosmetic} and the second equality was observed in \cite[Remark 1.1]{LRS18}.
This is a contradiction.
\end{proof}

\begin{proof}[Proof of \cref{thm: closed commutator}]
Set $N=n-10$.
First, note that we have the isomorphism
\begin{align}
\label{eq: H2 decomp}
H^2(X;\Z) \cong 2H\oplus (-E_8) \oplus N(-1).
\end{align}
Let $\fraks$ be the spin$^c$ structure on $X$ determined by
\[
c_1(\fraks) = (0,0,-1, \ldots, -1)
\]
along this decomposition, where the first two zeros are in $2H$ and $-E_8$ respectively.

Using the decomposition \eqref{eq: H2 decomp},
define $A_i \in \Aut(H_2(X;\Z))$ for $i=1,2$ by
\begin{align*}
A_1 &= (-\id_{H}) \oplus \id_{H} \oplus \id_{-E_8} \oplus \id_{N(-1)},\\
A_1 &= \id_{H} \oplus (-\id_{H}) \oplus \id_{-E_8} \oplus \id_{N(-1)}.
\end{align*} 
By Wall's theorem~\cite[Theorem~2]{Wa64},
there are $f_i \in \Diff(X)$ such that $(f_i)_\ast=A_i$.
Note that both $f_i$ preserve $\fraks$.

Set $f=[f_1,f_2]$.
Suppose that $f$ is smoothly isotopic to the identity.
Then it follows from \cref{thm: T2 spinc pres closed} that $c_1(\fraks)^2 -\sigma(X) \leq 0$, but this is a contradiction by the definition of $\fraks$.
Thus we have that $f$ is not smoothly isotopic to the identity.

On the other hand, since $f_\ast=1$, it follows from a result of Quinn~\cite{Q86} and Perron~\cite{P86} that $f$ is topologically isotopic to the identity.
This completes the proof.
\end{proof}

\begin{rem}[Isotopy vs. Pseudo-isotopy]
\label{rem pseudo isotopy}
All exotic diffeomorphisms detected in this paper on simply-connected 4-manifolds $X$, with or without boundary, are smoothly pseudo-isotopic to the identity, relative to the boundary if $\del X \neq \emptyset$.
This follows from a general result by Kreck~\cite[Theorem~1]{Kreck} for closed 4-manifolds, and a result by Orson and Powell \cite{OrsonPowell2022}, a combination of \cite[Theorem~2.6]{OrsonPowell2022} and \cite[Theorem~7.3]{OrsonPowell2022}, for 4-manifolds with (at least homology 3-sphere) boundary.

Thus our results also detect the difference between smooth isotopy and smooth pseudo-isotopy, as well as other known exotic diffeomorphisms detected by gauge theory. 
See \cite{BD19,Wa20,igusa2021second,Singh21} for other work detecting such differences for non-simply-connected 4-manifolds, obtained by tools different from gauge theory.
\end{rem}

\section{Problems}

We close this paper with questions naturally arising from our results.
See also Question~\ref{ques S4}.
Recall that Milnor fibers $M(2,3,7)$ and $M(2,3,11)$ are smoothly embedded in $K3$.
In particular, $\Sigma(2,3,7)$ and $\Sigma(2,3,11)$ are embedded in $K3$.

\begin{ques}
\label{ques: K3}
Is the Dehn twist on $K3$ along $\Sigma(2,3,7)$ or $\Sigma(2,3,11)$ exotic?
What about the Dehn twist on another irreducible 4-manifold in which $\Sigma(2,3,6n+7)$ or $\Sigma(2,3,6n+11)$ is embedded?
\end{ques}

Note that we have confirmed in \cref{thm: intro stabilization in K3} that the Dehn twists along $\Sigma(2,3,7)$ and $\Sigma(2,3,11)$ are {\it relatively} exotic on $\mathring{K3}$, the punctured $K3$, more strongly on $\mathring{K3}\# S^2 \times S^2$, when we puncture in the complement of Milnor fibers in $K3$.

The next question is about  Milnor fibers.
The Milnor fiber $M(2,3,5)$ can be constructed by an equivariant plumbing discussed in \cite{Orlik-Philip}, and it implies that the standard circle action on $\Sigma(2,3,5)$ extends to $M(2,3,5)$ as a smooth circle action.
On the other hand, we have seen in \cref{thm pi 1 Milnor necleus} that the Seifert loops $\gamma_S \in \pi_1(\Diff(Y))$ for $Y=\Sigma(2,3,7)$ and $\Sigma(2,3,11)$ do not smoothly extends to the corresponding Milnor fibers. 

\begin{ques}
\label{ques: Milnor fiber}
Let $(p,q,r) \neq (2,3,5)$.
Then does the circle action on $\Sigma(p,q,r)$ smoothly extend to $M(p,q,r)$?
More strongly, does the Seifert loop $\gamma_S \in \pi_1(\Diff(\Sigma(p,q,r)))$ smoothly extend to $M(p,q,r)$ in the sense of \cref{def loop ext}?
\end{ques}

By \cref{lem: exact}, this is equivalent to asking the non-triviality of the boundary Dehn twists of Milnor fibers in the relative mapping class groups.

\bibliographystyle{plain}
\bibliography{tex}

\end{document}